[1994/12/01]
\documentclass{ijmart}
\chardef\bslash=`\\ 

        \usepackage {amssymb,amsthm,amsxtra,amscd,amsmath,amsfonts,mathtools}
\usepackage{tikz}
        \newcommand{\C}{\ensuremath{\mathbb{C}}}
        \newcommand{\K}{\ensuremath{\Bbbk}}
        
        \newcommand{\Q}{\ensuremath{\mathbb{Q}}}
        
        \newcommand{\Z}{\ensuremath{\mathbb{Z}}}

	\newcommand{\rbracket}{]}

        \newcommand{\Hom}{\textup{Hom}}

        \newcommand{\del}{\ensuremath{\partial}}

        \theoremstyle{plain}
                
                \newtheorem{proposition}{Proposition}
              \newtheorem{remark}{Remark}
              
	     \newtheorem{theorem}{Theorem}

             \newtheorem*{proposition*}{Proposition}
                \newtheorem*{theorem*}{Theorem}
        \theoremstyle{definition}
                
        \numberwithin{equation}{section}
        
        \newcommand{\ignore}[1]{}
        
        \newcommand{\mynote}[1]{}

\begin{document}

\title[Representations of $D_q(\Bbbk  [x \rbracket )$]{Representations of $D_q(\Bbbk [x \rbracket )$}
\author[V. Futorny]{Vyacheslav Futorny}
\address{%
Departamento de Matem\' atica,
Universidade de S\~ao Paulo,
Caixa Postal 66281,
S\~ao Paulo-SP, Brazil
05315-970.}

\email{futorny@ime.usp.br}
\author[U. N. Iyer]{Uma N. Iyer}
\address{
Department of Mathematics and Computer Science,
Bronx Community College,
2155 University Avenue,
Bronx, New York 10453, USA.}
\email{uma.iyer@bcc.cuny.edu \thanks{Corresponding author.}}

\date{Received on September, 2014}
\begin{abstract}
The algebra of quantum differential operators on graded algebras was introduced by V. Lunts and A. Rosenberg.   D. Jordan, T. McCune and the second author have identified
 this algebra of quantum differential operators on the polynomial algebra with coefficients in an algebraically closed field of characteristic zero.
It contains   the first Weyl algebra and the quantum Weyl algebra as its subalgebras.
In this paper we classify irreducible weight modules over the algebra of quantum differential operators on the polynomial algebra.  
Some classes of indecomposable modules are constructed in the case of positive characteristic and $q$ root of unity.
  
\end{abstract}
\maketitle
\tableofcontents
\section{Introduction}
Let $\K = \Q (q)$ be a field where $q \in \C^*$ is not a root of unity. 
Let $\K [x]$ denote the polynomial algebra over $\K$ in one variable $x$.
Then following \cite{LR}, the algebra of quantum differential operators on $\K [x]$, denoted by
$D_q(\K [x])$, for a specific bicharacter $\beta$ was constructed in \cite{IM}.  In \cite{IJ}, the algebra
$D_q(\K [x])$ was described in terms of generators and relations.

View $x$ as a linear map on $\K [x]$ via multiplication; that is, 
$x (f) = xf$ for $f\in \K [x]$. 
For $a\in \{ 1,-1 \}$ define $\del_a : \K [x] \to \K [x]$ linearly by 
$\del_a (x^n) = \left( \dfrac{q^{an}-1}{q^{a}-1} \right) x^{n-1}$ for $n\geq 1$ and
$\del_a (1)=0$. Further, define $\del_0 = \del : \K [x] \to \K [x]$ linearly by 
$\del (x^n) = nx^{n-1}$ for $n\geq 0$. 
Then, $D_q(\K [x])$ is the subalgebra of the algebra of $\K$-linear homomorphisms, $\Hom (\K [x], \K[x])$,
generated by the maps $\{x, \del_1, \del_{-1}, \del_0 \}$ (\cite{IM}). 
The defining relations (\cite{IJ}) among these maps are
\begin{align}
\del_a x - q^a x \del_a &=1,\label{Dq-relation-1}\\
\del_a x \del_b &= \del_b x \del_a,\label{Dq-relation-2}\\
\del_{-1} \del_1 &= q \del_1 \del_{-1}, \quad \textit{for } a,b \in \{ 1,-1,0 \}.\label{Dq-relation-3}
\end{align}
In \cite{IJ}, it is shown that $D_q(\K [x])$ is a left and right Noetherian simple domain of GK dimension 3 when
$q$ is transcendental over $\Q$. 

Let $\sigma : \K [x] \to \K [x]$ be the automorphism given by $\sigma (x^n) = q^n x^n$ for $n\geq 0$.
Then one can check that 
\begin{align}
\del_1 x - x\del_1 &=\sigma = (q-1)x\del_1 +1, \label{sigma-as-commutator}\\
\del_{-1} x - x\del_{-1} &= \sigma^{-1} = (q^{-1}-1) x\del_{-1} +1,\label{sigma-inverse-as-commutator}\\
\del_{-1} &= \sigma^{-1} \del_1. \label{del-minus-one} 
\end{align}
Note that $D_q ( \K [x])$ is $\Z$-graded with $deg (x)=1$, $deg (\sigma) = deg (\sigma^{-1}) = 0$ and $deg (\del) = deg (\del_1) = -1$. 
Let 
$\tau = \del x$. Then, we have $\tau \sigma = \sigma \tau$ and  
$ \K [\tau, \sigma,\sigma^{-1}]$ is the subalgebra of $D_q( \K [x])$ of degree $0$. Note that 
$\K [\tau , \sigma]$ is
indeed a polynomial algebra, and 
$ \K [\tau, \sigma,\sigma^{-1}]$ is the localization of $\K [\tau , \sigma]$ at powers of $\sigma$.
Let $\alpha : \K [\tau, \sigma,\sigma^{-1}]\to \K [\tau, \sigma,\sigma^{-1}]$ be an algebra automorphism given by $\alpha (\tau ) = \tau -1$ and
$\alpha (\sigma ) = \dfrac{\sigma}{q}$. 

By \eqref{Dq-relation-1}-\eqref{Dq-relation-3} and \eqref{sigma-as-commutator}-\eqref{del-minus-one} we may view $D_q(\K[x])$ as an algebra generated over $ \K [\tau, \sigma,\sigma^{-1}]$ by generators $\{ x, \del, \del_1 \}$ subject to the relations
\begin{align}
\del x &= \tau, \quad x \del = \tau-1 ,\\
\del_1 x&= \dfrac{q\sigma -1}{q-1}, \quad x \del_1 = \dfrac{\sigma -1}{q-1} ,\\
\del_1 (\tau -1)  &= \del \left( \dfrac{\sigma -1}{q-1} \right)  , \\
x \tau &= (\tau -1) x, \quad x \sigma = \dfrac{\sigma}{q}x, \\
\del \tau &= (\tau +1) \del, \quad \del \sigma = q\sigma \del, \\
\del_1 \tau &= (\tau +1) \del_1, \quad \del_1 \sigma = q\sigma \del_1. 
\end{align}
This motivates the following definition/notations.
\subsection*{Definition/Notations}
Let $\K$ be an arbitrary field, $q\in \K^*$. For variables $\tau, \sigma$, let 
$R= \K [\tau, \sigma,\sigma^{-1}]$ be the localization of the polynomial algebra $\K [ \tau,\sigma]$
with respect to the multiplicative set $\{ \sigma^i \}_{\i\in \Z_{\geq 0} }$.   Let
$\alpha: R \to R$ be the  automorphism
given by $\alpha (\tau ) = \tau -1$ and
$\alpha (\sigma ) = \dfrac{\sigma}{q}$. 
Let $D$ be the algebra generated over $R$ with variables $X,Y,Y_1$ satisfying relations
\begin{align}
YX &= \tau, \quad XY = \alpha (\tau), \\
Y_1X &= q\sigma -1, \quad XY_1 = \alpha (q\sigma -1), \\
Y_1 (\tau-1) &= Y  (\sigma -1),\\
Xr=\alpha (r)X, \quad Y r &= \alpha^{-1}(r)Y, \quad Y_1 r = \alpha^{-1}(r)Y_1,
\forall r\in R.
\end{align}
Our goal is to study weight modules of $D$ along the lines of \cite{BBF} and \cite{DGO}
using the fact that $D$ contains two Generalized Weyl Algebras (GWA), denoted by
$A_q, A_1$. 
In this paper we deal only with the problem of classifying irreducible modules when $\K$ is algebraically closed. 
We will treat the indecomposable ones in a subsequent paper. 

We show that all irreducible modules in these cases come in three groups:\\

\noindent
\textbf{Family I:} Irreducible $D$ weight modules which are extended from irreducible $A_q$-weight modules (see section \ref{DGO-indecomposables-of-Aq-char0}).

\noindent
 \textbf{Family II:} Irreducible $D$ weight modules which are extended from irreducible $A_1$-weight modules (see section \ref{DGO-indecomposables-of-A1-char0}).

\noindent
\textbf{Family III:} Irreducible
$D$-weight modules which do not arise from the families I and II. These are indecomposable as $A_q$ and as $A_1$ modules (see sections \ref{DGO-indecomposables-of-Aq-char0} and \ref{DGO-indecomposables-of-A1-char0}) by our main theorem,

\begin{theorem*}[\ref{irred-D-to-indecom}]
When $\K$ is algebraically closed, every irreducible $D$ weight module is indecomposable as an 
$A_q$-module and as an $A_1$-module.
\end{theorem*}
 
Theorem \ref{alg.cl.char0.Virred}, Theorem \ref{alg.cl.q-not-root.Virred}, and 
Theorem \ref{alg.cl.finite.orbit}
are combined to 
\begin{theorem*}[B]
The irreducible $D$-modules are as described in
families I, II, and III.
\end{theorem*}
In section \ref{indecomposablefamily}, we build a family of indecomposable $D$-modules
which are decomposable as  $A_q$ and $A_1$ modules (see Theorem \ref{thm-family-pm})
when the characteristic of $\K$ is equal to the order of $q$. In this same set-up, we present,

\begin{theorem*}[\ref{thm-equidimension-p=|q|}]
Let $V = \oplus_{\mathfrak{m} \in \omega} V_{\mathfrak{m}}$ 
be a finite dimensional $D$ weight module with support contained in orbit $\omega$. 
Then $dim (V_{\mathfrak{m}}) = dim(V_{\mathfrak{n}})$ for any $\mathfrak{m}, \mathfrak{n}
\in \omega$.
\end{theorem*}

In the process, we have investigated
which indecomposable weight modules of $A_q$ and $A_1$ can be extended to 
$D$ weight modules, which is of independent interest (sections \ref{DGO-indecomposables-of-Aq-char0}
and \ref{DGO-indecomposables-of-A1-char0}). In particular, we note that in many cases, it is not possible
to extend an indecomposable $A_q$-module (or an $A_1$-module) to a $D$-module.

Assume $\overline{\K} = \K$. For this reason, every irreducible polynomial which is needed
to describe weight modules is linear.
\section{Preliminaries}
Let $A_q$ be the $R$-subalgebra of $D$ generated by $X,Y_1$.  Then, $A_q$ is a 
GWA:
\begin{align}
Y_1X &= q\sigma -1, \quad XY_1 = \alpha (q\sigma -1), \\
Xr&=\alpha (r)X,  \quad Y_1 r = \alpha^{-1}(r)Y_1,
\forall r\in R.
\end{align}

The algebra $D$ is then generated over $A_q$ by $Y$ with relations
\begin{align}
YX &= \tau, \quad XY = \alpha (\tau), \\
Y_1 (\tau-1) &= Y  (\sigma -1),\\
 Y r &= \alpha^{-1}(r)Y, 
\forall r\in R.
\end{align}

Analogously, let $A_1$ be the $R$-subalgebra of $D$ generated by $X,Y$.  Then $A_1$ is
also a GWA with
\begin{align}
YX &= \tau, \quad XY = \alpha (\tau -1), \\
Xr&=\alpha (r)X,  \quad Y r = \alpha^{-1}(r)Y,
\forall r\in R.
\end{align}
 $D$ is generated over $A_1$ by $Y_1$ with relations
\begin{align}
Y_1X &= (q\sigma -1), \quad XY_1 = \alpha (q\sigma -1), \\
Y_1 (\tau-1) &= Y  (\sigma -1),\\
 Y_1 r &= \alpha^{-1}(r)Y_1, 
\forall r\in R.
\end{align}
In what follows, $A$ could mean $A_q$ or $A_1$, and
 $t=q\sigma -1$ (if we study weight modules over $A_q$) 
or $t=\tau$ (if we study weight-modules over $A_1$).

Let $Max(R)$ denote the set of maximal ideals of $R$. For $\mathfrak{m} \in Max(R)$,
set $M_{\mathfrak{m}} = \{ v\in M \mid \mathfrak{m}v=0 \}$. We say that $M$ is a weight module
if $M=\sum_{\mathfrak{m}\in Max(R)} M_{\mathfrak{m}}$.  Let support of $M$ be
$Supp(M) = \{ \mathfrak{m} \in Max(R) \mid
M_{\mathfrak{m}} \neq 0 \}$. The cyclic group $<\alpha>$ generated by $\alpha$
acts on the set $Max (R)$, and let $\Omega$ denote the corresponding orbit set.
When $\omega \in \Omega$ is infinite we call $\omega$ a \textit{linear orbit}. When
$\omega$ is finite it is called a \textit{circular orbit}. 
 Set an order on $\omega$ by
$\mathfrak{m}<\alpha (\mathfrak{m})$.

For $\mathfrak{m}\in Max(R)$, let $K_{\mathfrak{m}} = R/\mathfrak{m}$ and $t_{\mathfrak{m}} = 
t+\mathfrak{m} \in K_{\mathfrak{m}}$. Call $\mathfrak{m}$ a \emph{break} if
$t_{\mathfrak{m}} =0$. Let $B$ be the set of all breaks, and $B_{\omega} = B \cap \omega$ for
$\omega \in \Omega$. A \textit{maximal break} in a linear orbit is $\mathfrak{m} \in B$ which
is maximal with respect to the order $<$ defined in the preceding paragraph. For a circular orbit, chose any
break to be the maximal break. 

Note that \cite{DGO} gives lists of indecomposable and simple weight GWA modules.
\subsection{Description of indecomposable $A$-modules.}\label{general-set-up}
For each orbit $\omega$ of $r$ elements, 
fix a maximal ideal $\mathfrak{m}(\omega) \in \omega$
which is chosen to be a maximal 
break if $\omega$ has a break.   
In what follows, $T=Y_1$ if we are studying $A_q$-modules, and $T=Y$ if 
we are studying $A_1$-modules, and $A$ denotes $A_q$ or $A_1$.

\noindent
\textbf{Case $|\omega|=\infty, B_{\omega} = \emptyset$:}
Here, $V(\omega) =  \oplus_{\mathfrak{m}\in \omega} K_{\mathfrak{m}}$ has an irreducible $A$-module structure with 
$X(v)=\alpha (t_{\mathfrak{m}}v)$ and $T(v) = \alpha^{-1}(v)$ for $v\in \K_{\mathfrak{m}}$. 
The module $V (\omega)$ is irreducible.

\noindent
\textbf{Case $|\omega|=\infty, B_{\omega} \neq \emptyset$:}
If $\mathfrak{m}$ is the maximal element in $B_{\omega}$
then set $B_{\omega}' = B_{\omega} \cup \{ \alpha (\mathfrak{m}) \}$. If $B_{\omega}$ does not have a
maximal element, then $B_{\omega}' = B_{\omega}$.
Let $J \subset B_{\omega}'$ be an interval, and $J' \subset J$
be any subset not containing the maximal element of $J$.
Let $\mathfrak{n}_0$ be the maximal break in $\omega$ preceding all elements of $J$ or $-\infty$ if it does
not exist; let $\mathfrak{n}_1$ be the maximal element of $J$ if it exists and is a break or $+\infty$ otherwise.
For each $\mathfrak{m} \in \omega$, put $V_{\mathfrak{m} } = K_{\mathfrak{m}}$ if
$\mathfrak{n}_0 < \mathfrak{m} \leq \mathfrak{n}_1$ and $V_{\mathfrak{m}} = 0$ otherwise. 
Set $V(\omega, J, J') = \oplus_{\mathfrak{m} \in \omega} V_{\mathfrak{m}}$, and for $v\in V_{\mathfrak{m}}$ set 
\[
X(v) = \begin{cases}
		\alpha (t_{\mathfrak{m}}v) &\textit{if $\mathfrak{m} \notin B$,}\\
		\alpha (v) &\textit{if $\mathfrak{m} \in J'$,}\\
		0 &\textit{otherwise};
	\end{cases} \quad
T(v) = \begin{cases}
		0 &\textit{if $\alpha^{-1}(\mathfrak{m}) \in J'\cup \{ \mathfrak{n}_0 \}$,}\\
		\alpha^{-1}(v) &\textit{otherwise}.
	\end{cases}
\]

\noindent
\textbf{Case $|\omega|<\infty, B_{\omega} = \emptyset$:}
Fix any $f\in \K^*$. 
Let $V(\omega ,f) = \oplus_{\mathfrak{m}\in \omega} \K_{\mathfrak{m}}$ and 
give it $A$-module structure by setting 
\[
X(v) = \begin{cases} 
             \alpha (t_{\mathfrak{m}} v) &\textit{ if } \mathfrak{m} \neq \mathfrak{m}(\omega),\\
             f \alpha (t_{\mathfrak{m}} v) &\textit{ if } \mathfrak{m} = \mathfrak{m}(\omega);
		\end{cases}\quad
T(v) = \begin{cases} 
             \alpha^{-1} (v) &\textit{ if } \alpha ^{-1}(\mathfrak{m} )\neq \mathfrak{m}(\omega),\\
             \frac{1}{f} \alpha^{-1}( v) &\textit{ if } \alpha^{-1}(\mathfrak{m}) = \mathfrak{m}(\omega).
		\end{cases}
\]
The module $V(\omega ,f)$ is irreducible.

\noindent
\textbf{Case $|\omega|<\infty, B_{\omega} \neq \emptyset$:} Here, we receive two families of indecomposable 
families. Suppose $|B_{\omega}| = m>0$. Then there is a one-to-one correspondence $\Z_m \to B_m$
with $\mathfrak{m}_i = \alpha^i (\mathfrak{m}_0)$ for some fixed break $\mathfrak{m}_0$. For any $\mathfrak{m} \in \omega$, let $j(\mathfrak{m})
\in \Z_m$ be the unique $j$ such that $\mathfrak{m}_{j-1} < \mathfrak{m} \leq \mathfrak{m}_j$. Let $x,y$ be
two noncommuting variables.  

\noindent
\textbf{Family 1.}
Fix $j\in \Z_m$ and $w=z_1z_2\cdots z_n$ be a word of length 
$n\geq 1$ where each $z_i \in \{ x, y\}$.  Let $e_0,e_1,\ldots, e_n$ be $n+1$ symbols. For each $\mathfrak{m}\in \omega$, let $V_{\mathfrak{m}}$ be a vector space over $\K_{\mathfrak{m}}$ with basis
$\{ (\mathfrak{m}, e_k) \mid k+j = j(\mathfrak{m}) \in \Z_m \}$. Put 
$V(\omega ,j, w) = \oplus_{\mathfrak{m} \in \omega} V_{\mathfrak{m}}$ and it has an $A$-module 
structure given by  
\[
X (\mathfrak{m},e_k) = \begin{cases} t_{\mathfrak{m}} (\alpha (\mathfrak{m}), e_k) &\textit{ if } \mathfrak{m} \notin B,\\
							(\alpha (\mathfrak{m}), e_{k+1}) &\textit{ if } \mathfrak{m} \in 
									B_{\omega} 
										\textit{ and } z_{k+1}=x,\\ 
							0 &\textit{ otherwise};  
				\end{cases}
\]
\[
T (\mathfrak{m},e_k) = \begin{cases}  (\alpha^{-1} (\mathfrak{m}), e_k) &\textit{ if } \mathfrak{m} 
										\textit{ is not a break},\\
							(\alpha^{-1} (\mathfrak{m}), e_{k-1}) &\textit{ if } \mathfrak{m} \in B_{\omega} 
										\textit{ and } z_{k}=y,\\ 
							0 &\textit{ otherwise}.  
				\end{cases}
\]
Let $\mathfrak{m}_1$ be that break for which $n+j = j(\mathfrak{m}_1) \in \Z_m$. 
Then $X(\mathfrak{m}_1, e_n)=0$. Similarly, let $\mathfrak{m}_0$ be that break
for which $j=j(\mathfrak{m}_0) \in \Z_m$. Then
$T(\mathfrak{m}_0, e_0)=0$.
The module $V(\omega ,j, w) $ is irreducible if and only if $w$ is the empty word. 

\noindent
\textbf{Family 2.}
Let $w=z_1z_2\cdots z_n$  be a word whose length $n$ is a multiple of $m$ for $z_i \in \{ x, y\}$. 
Let $f\in \K^*$. Consider $n$ elements $e_{k}$ where
$k=1,2,\ldots ,n$. For $\mathfrak{m}\in \omega$, let $V_{\mathfrak{m}}$ be a $\K_{\mathfrak{m}}$ vector
space with basis $\{ e_{k} \mid k \equiv j(\mathfrak{m}) (\textit{mod }m) \}$. The vector space 
$V(\omega ,w, f) = \oplus_{\mathfrak{m}\in \omega} V_{\mathfrak{m}}$ has an $A$-module structure given by
\[
X(\mathfrak{m},e_{k}) = \begin{cases} 
                                      t_{\mathfrak{m}}( \alpha( \mathfrak{m}), e_{k} ) 
						&\textit{ if }\mathfrak{m} \textit{ is not a break;}\\
				 ( \alpha( \mathfrak{m}), e_{k+1} ) 
						&\textit{ if }\mathfrak{m} \in B_{\omega}, k \neq n, z_{k+1} =x;\\
				f( \alpha( \mathfrak{m}), e_{1} ) 
						&\textit{ if }\mathfrak{m} \in B_{\omega}, k = n, z_{1} =x;\\
				0 &\textit{ otherwise;} 
				\end{cases}
\]
\[
T(\mathfrak{m},e_{k}) = \begin{cases} 
                                      ( \alpha ^{-1}( \mathfrak{m}), e_{k} ) 
						&\textit{ if }\alpha^{-1}(\mathfrak{m}) \textit{ is not a break;}\\
				 ( \alpha ^{-1}( \mathfrak{m}), e_{k-1} ) 
						&\textit{ if }\alpha^{-1}((\mathfrak{m} )\in B_{\omega}, k \neq 1, z_{k} =y;\\
				f( \alpha ^{-1}( \mathfrak{m}), e_{n} ) 
						&\textit{ if }\alpha^{-1}(\mathfrak{m}) \in B_{\omega}, k = 1, z_{1} =y;\\
				0 &\textit{ otherwise;} 
				\end{cases}
\]
The module $V(\omega, w,f)$ is irreducible if and only if $w =x^{m}$ or $w=y^{m}$.

\section{Indecomposable modules over $A_q$ and $A_1$.}
\subsection{The list of $A_q$-indecomposables.}\label{DGO-indecomposables-of-Aq-char0}
Here we go through the list of GWA-indecomposables as
described in \cite{DGO} and understand them as $A_q$-weight modules.

\subsubsection{The module $V_q(\omega ,b)$ for $b\in \K^* \setminus \{ q^i \}_{i\in \Z}$.}\label{qVomega}
Suppose that $|\omega|=\infty$ and $B_{\omega}= \emptyset$.
In this case,  we see $V_q(\omega, b) = \oplus_{i\in \Z} \K v_i$ where $\{ v_i \}_{i \in \Z}$ is a basis,
and the action is defined by:
\begin{align*}
&X(v_i)= (q^{i+1}b-1)v_{i+1},
&Y_1(v_i) = v_{i-1},\\
&f(\tau, \sigma) (v_i) = f (a+i, q^ib) v_i,&\textit{for } f\in \K [\tau, \sigma, \sigma^{-1}].
\end{align*}
Note in particular that $(\sigma - 1) (v_i) = (q^ib -1)v_i \neq 0$ since $b \notin \{ q^j \}_{j\in \Z}$.
Thus, in this case we can extend the action of $A_q$ on $V_q(\omega ,b)$ to an action of $D$ by setting
\[
Y (v_i) = \dfrac{(a+i-1)}{(q^ib-1)} v_{i-1} \quad \forall i \in \Z.
\]
Thus, for any $a\in \K$, we obtain an irreducible $D$-weight module, denoted by $V_q(\omega ,b,a)$.
Using appropriate scalar multiples of $v_i$, we visualize $V_q(\omega ,b,a)$ without loss of generality as follows:

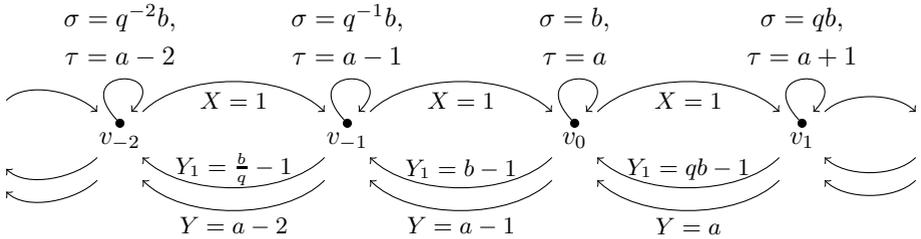
\begin{center}
\begin{tikzpicture}[scale=1.5]
\coordinate [label=below:$v_{-2}$] (A) at (0,0);
\coordinate [label=below:$v_{-1}$] (B) at (2,0);
\coordinate [label=below:$v_0$] (C) at (4,0);
\coordinate [label=below:$v_1$] (D) at (6,0);
\filldraw[black](A) circle (1pt);
\filldraw[black](B) circle (1pt);
\filldraw[black](C) circle (1pt);
\filldraw[black](D) circle (1pt);

\draw [->] (A)  .. controls  (-0.5,0.5) and (0.5,0.5)  ..   (0.1, 0.1) ;
\node (la) at (0,0.75) {$\begin{matrix}\sigma = q^{-2}b,\\ \tau =a-2 \end{matrix}$};

\draw [->] (B)  .. controls  (1.5,0.5) and (2.5,0.5)  ..   (2.1, 0.1) ;
\node (lb) at (2,0.75) {$\begin{matrix}\sigma = q^{-1}b,\\ \tau =a-1 \end{matrix}$};

\draw [->] (C)  .. controls  (3.5,0.5) and (4.5,0.5)  ..   (4.1, 0.1) ;
\node (lc) at (4,0.75) {$\begin{matrix}\sigma = b,\\ \tau =a \end{matrix}$};

\draw [->] (D)  .. controls  (5.5,0.5) and (6.5,0.5)  ..   (6.1, 0.1) ;
\node (ld) at (6,0.75) {$\begin{matrix}\sigma = qb,\\ \tau =a+1 \end{matrix}$};


\draw[->] (-1,0.2) to [distance=0.25cm](-0.2,0.1);
\draw[->] (-0.2,-0.3)to [distance=-0.25cm] (-1,-0.4) ;
\draw[->] (-0.2,-0.5)to [distance=-0.25cm] (-1,-0.6) ;
\draw[->] (0.2,0.1) to [distance=0.5cm](1.8,0.1);
\draw[->] (1.8,-0.3) to [distance=-0.5cm](0.2,-0.3);
\draw[->] (1.8,-0.5) to [distance=-0.5cm](0.2,-0.5);
\node (xa) at (1,0.2) {\small{$X=1$}};
\node (y1b) at (1,-0.4) {\small{$Y_1=\frac{b}{q}-1$}};
\node (y1b) at (1,-0.9) {\small{$Y=a-2$}};

\draw[->] (2.2,0.1) to [distance=0.5cm](3.8,0.1);
\node (xb) at (3,0.2) {\small{$X=1$}};
\draw[->] (3.8,-0.3) to [distance=-0.5cm](2.2,-0.3);
\draw[->] (3.8,-0.5) to [distance=-0.5cm](2.2,-0.5);
\node (y1b) at (3,-0.4) {\small{$Y_1=b-1$}};
\node (y1b) at (3,-0.9) {\small{$Y=a-1$}};

\draw[->] (4.2,0.1) to [distance=0.5cm](5.8,0.1);
\node (xc) at (5,0.2) {\small{$X=1$}};
\draw[->] (5.8,-0.3) to [distance=-0.5cm](4.2,-0.3);
\draw[->] (5.8,-0.5) to [distance=-0.5cm](4.2,-0.5);
\node (y1c) at (5,-0.4) {\small{$Y_1=qb-1$}};
\node (y1c) at (5,-0.9) {\small{$Y=a$}};

\draw[->] (6.2,0.1) to [distance=0.25cm](7,0.1);
\draw[->] (7,-0.3) to [distance=-0.25cm](6.2,-0.3);
\draw[->] (7,-0.5) to [distance=-0.25cm](6.2,-0.5);

\end{tikzpicture}

\end{center}

\begin{remark}\label{rem:Vqomega}
When characteristic of $\K$ is $0$, and $a\in \K \setminus \Z$ (respectively, $a\in \K \setminus \Z_p$
when characteristic of $\K$ is $p>0$)
then the resulting module $V_q(\omega ,b,a)$
is an irreducible $A_1$-module. If $a\in \Z$ (respectively, $a\in \Z_p$),
then the resulting module $V_q(\omega, b,a)$ is an indecomposable $A_1$-module. 
\end{remark}
\subsubsection{The module $V_q(\omega ,J, J')$ when $q$ is not a root of unity.}\label{qVJJ'}
Suppose $|\omega|= \infty$ and $B_{\omega} \neq \emptyset$.
That is, $\omega = \{ \alpha ^i (\tau-a, \sigma -b) \}_{i\in \Z}$.
Since $B=\{ (\tau -c, q\sigma -1) \mid c\in \K \}$, we 
have $B_{\omega} \neq \emptyset$ if and only if $b = q^k$ for some $k\in \Z$.
That is, without loss of generality, 
\[
\omega = \{  \mathfrak{m}_i = (\tau -a -i, \sigma - q^{i-1}) \}_{i\in \Z},
 \textit{ and }
B_{\omega} = \left\{  \mathfrak{m}_{0} = \left( \tau -a, \sigma  -\frac{1}{q} \right) \right\}.
\]
Set $B_{\omega}' = \{ \mathfrak{m}_{0}, \mathfrak{m}_{1} \}$. 
Thus, we have options which we investigate one at a time:

\noindent
(1)
  $J = B_{\omega}', J' = \{ \mathfrak{m}_{0} \}$, $\mathfrak{n}_0 =-\infty, \mathfrak{n}_1 =\infty$.
In this case, we have $V = \oplus_{i=-\infty}^{\infty} \K v_i$.  Set $\K v_i = K_{\mathfrak{m}_{i}}$ 
for $i\in \Z$. That is, $v_0 \in K_{\mathfrak{m}_{0}}$. 
Here, 
\[
X(v) = \begin{cases}
	v_1 &\textit{if }v = v_0\\
	(q^i-1)v_{i+1}&\textit{if }v = v_i
	\end{cases}
\quad \quad
Y_1(v) = \begin{cases}
	0 &\textit{if }v = v_1\\
	v_{i-1}&\textit{if }v = v_i
	\end{cases}\quad \forall i\in \Z.
\]
Further, $\sigma (v_i) = q^{i-1}v_i, \quad \tau (v_i) = (a+i)v_i \quad \forall i\in \Z$.

Using the fact that $YX = \tau$,  
we now define action of $Y$ on $V_q( \omega , J,J')$ as ,
\[
Y(v_{i+1}) = \dfrac{(a+i)}{(q^i-1)} v_i \quad \forall i\neq 0, \quad
Y(v_1)  = av_0.
\]
For every $a\in \K$, we thus have a $D$ weight module denoted by $V_q(\omega, J, J',a)$ which is
irreducible if $a\neq 0$, and indecomposable if $a=0$. Note that if $a\in \Z \setminus \{ 0 \}$ then
$V_q(\omega , J,J',a)$ is not simple as either $A_q$ module or $A_1$-module. But it is simple as a $D$-module. 
By considering suitable scalar multiples of $v_i$, without loss of generality, we visualize $V_q(\omega ,J,J',a)$ as follows:

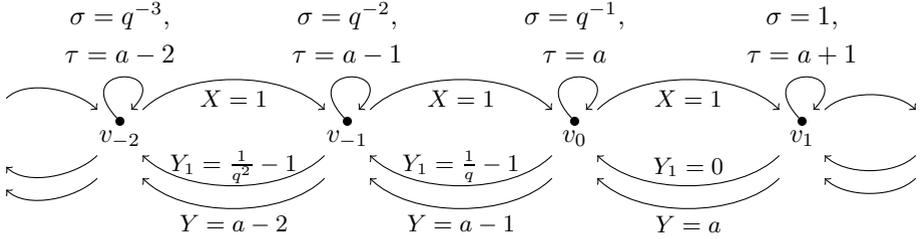
\begin{center}
\begin{tikzpicture}[scale=1.5]
\coordinate [label=below:$v_{-2}$] (A) at (0,0);
\coordinate [label=below:$v_{-1}$] (B) at (2,0);
\coordinate [label=below:$v_0$] (C) at (4,0);
\coordinate [label=below:$v_1$] (D) at (6,0);
\filldraw[black](A) circle (1pt);
\filldraw[black](B) circle (1pt);
\filldraw[black](C) circle (1pt);
\filldraw[black](D) circle (1pt);

\draw [->] (A)  .. controls  (-0.5,0.5) and (0.5,0.5)  ..   (0.1, 0.1) ;
\node (la) at (0,0.75) {$\begin{matrix}\sigma = q^{-3},\\ \tau =a-2 \end{matrix}$};

\draw [->] (B)  .. controls  (1.5,0.5) and (2.5,0.5)  ..   (2.1, 0.1) ;
\node (lb) at (2,0.75) {$\begin{matrix}\sigma = q^{-2},\\ \tau =a-1 \end{matrix}$};

\draw [->] (C)  .. controls  (3.5,0.5) and (4.5,0.5)  ..   (4.1, 0.1) ;
\node (lc) at (4,0.75) {$\begin{matrix}\sigma = q^{-1},\\ \tau =a \end{matrix}$};

\draw [->] (D)  .. controls  (5.5,0.5) and (6.5,0.5)  ..   (6.1, 0.1) ;
\node (ld) at (6,0.75) {$\begin{matrix}\sigma = 1,\\ \tau =a+1 \end{matrix}$};


\draw[->] (-1,0.2) to [distance=0.25cm](-0.2,0.1);
\draw[->] (-0.2,-0.3)to [distance=-0.25cm] (-1,-0.4) ;
\draw[->] (-0.2,-0.5)to [distance=-0.25cm] (-1,-0.6) ;
\draw[->] (0.2,0.1) to [distance=0.5cm](1.8,0.1);
\draw[->] (1.8,-0.3) to [distance=-0.5cm](0.2,-0.3);
\draw[->] (1.8,-0.5) to [distance=-0.5cm](0.2,-0.5);
\node (xa) at (1,0.2) {\small{$X=1$}};
\node (y1b) at (1,-0.4) {\small{$Y_1=\frac{1}{q^2}-1$}};
\node (y1b) at (1,-0.9) {\small{$Y=a-2$}};

\draw[->] (2.2,0.1) to [distance=0.5cm](3.8,0.1);
\node (xb) at (3,0.2) {\small{$X=1$}};
\draw[->] (3.8,-0.3) to [distance=-0.5cm](2.2,-0.3);
\draw[->] (3.8,-0.5) to [distance=-0.5cm](2.2,-0.5);
\node (y1b) at (3,-0.4) {\small{$Y_1=\frac{1}{q}-1$}};
\node (y1b) at (3,-0.9) {\small{$Y=a-1$}};

\draw[->] (4.2,0.1) to [distance=0.5cm](5.8,0.1);
\node (xc) at (5,0.2) {\small{$X=1$}};
\draw[->] (5.8,-0.3) to [distance=-0.5cm](4.2,-0.3);
\draw[->] (5.8,-0.5) to [distance=-0.5cm](4.2,-0.5);
\node (y1c) at (5,-0.4) {\small{$Y_1=0$}};
\node (y1c) at (5,-0.9) {\small{$Y=a$}};

\draw[->] (6.2,0.1) to [distance=0.25cm](7,0.1);
\draw[->] (7,-0.3) to [distance=-0.25cm](6.2,-0.3);
\draw[->] (7,-0.5) to [distance=-0.25cm](6.2,-0.5);

\end{tikzpicture}

\end{center}

\noindent
(2)  $J = B_{\omega}', J' = \emptyset $, $\mathfrak{n}_0 =-\infty, \mathfrak{n}_1 =\infty$.
In this case again, we have $V = \oplus_{i=-\infty}^{\infty} \K v_i$.  Set $\K v_i = K_{\mathfrak{m}_{i}}$ 
for $i\in \Z$. 
Here, 
\[
X(v_i) = (q^i-1)v_{i+1},
\quad 
Y_1(v_i) = v_{i-1},
\quad
\sigma (v_i) = q^{i-1}v_i, \quad \tau (v_i) = (a+i)v_i,  \forall i\in \Z.
\]
Using the fact that $YX = \tau$, we may define $Y$ on $V_q( \omega , J,J')$, but
requiring that $XY (v_1) = (\tau -1)(v_1)$ gives us $a = 0$.  Thus, we have the following:
\[
\tau (v_i) = iv_i; \quad 
Y(v_i) = \begin{cases}
			cv_0 &\textit{if }i =1, \textit{ for some }c\in \K,\\
			\left( \dfrac{i-1}{q^{i-1}-1} \right) v_{i-1} &\textit{if } i\neq 1.	
		\end{cases}
\]
For any fixed $c\in \K$, we have a $D$-module which is indecomposable, but not irreducible. 
Notice that if $c=0$, then $V$ is a decomposable $A_1$-module, but an 
indecomposable $A_q$-module which is reducible as an $A_1^q$-module. 

\noindent
\textbf{A generalization of the above set-up:}
For any $c,d\in \K$, set $V_q(\omega ,J, J', c,d) = \oplus_{i=-\infty}^{\infty} \K v_i$ with
Here, 
\[
X(v_i) = (q^i-1)v_{i+1};
\quad 
\sigma (v_i) = q^{i-1}v_i, \quad \tau (v_i) = iv_i \quad \forall i\in \Z.
\]
\begin{align*}
Y_1(v_i) &= \begin{cases}
			cv_0 &\textit{if }i =1, \textit{ for some }c\in \K,\\
			 v_{i-1} &\textit{if } i\neq 1.	
		\end{cases} \\
Y(v_i) &= \begin{cases}
			dv_0 &\textit{if }i =1, \textit{ for some }d\in \K,\\
			\left( \dfrac{i-1}{q^{i-1}-1} \right) v_{i-1} &\textit{if } i\neq 1.	
		\end{cases}
\end{align*}
We have a $D$-module which is indecomposable (but not irreducible) 
for $(c,d)\neq (0,0)$.  If
$c=0$ and $d\neq 0$, then we have a module which is decomposable as an $A_1^q$-module,
but indecomposable (but not simple) as an an $A_1$-module and hence indecomposable as a $D$-module. Conversely, if $c\neq 0$ and $d=0$, then we have a module which is decomposable as an $A_1$-module
but indecomposable (yet not simple) as an $A_1^q$-modue and hence indecomposable as a $D$-module. 
With a suitable change of basis, the picture is as follows:

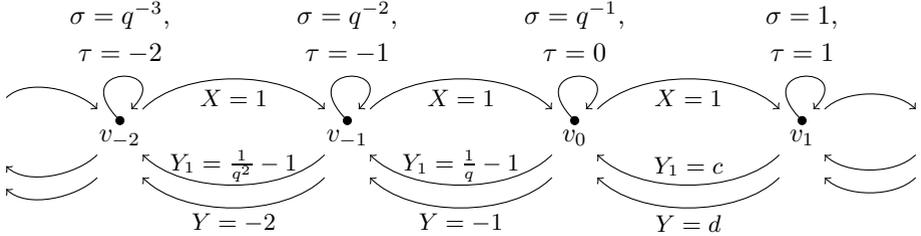
\begin{center}
\begin{tikzpicture}[scale=1.5]
\coordinate [label=below:$v_{-2}$] (A) at (0,0);
\coordinate [label=below:$v_{-1}$] (B) at (2,0);
\coordinate [label=below:$v_0$] (C) at (4,0);
\coordinate [label=below:$v_1$] (D) at (6,0);
\filldraw[black](A) circle (1pt);
\filldraw[black](B) circle (1pt);
\filldraw[black](C) circle (1pt);
\filldraw[black](D) circle (1pt);

\draw [->] (A)  .. controls  (-0.5,0.5) and (0.5,0.5)  ..   (0.1, 0.1) ;
\node (la) at (0,0.75) {$\begin{matrix}\sigma = q^{-3},\\ \tau =-2 \end{matrix}$};

\draw [->] (B)  .. controls  (1.5,0.5) and (2.5,0.5)  ..   (2.1, 0.1) ;
\node (lb) at (2,0.75) {$\begin{matrix}\sigma = q^{-2},\\ \tau =-1 \end{matrix}$};

\draw [->] (C)  .. controls  (3.5,0.5) and (4.5,0.5)  ..   (4.1, 0.1) ;
\node (lc) at (4,0.75) {$\begin{matrix}\sigma = q^{-1},\\ \tau =0 \end{matrix}$};

\draw [->] (D)  .. controls  (5.5,0.5) and (6.5,0.5)  ..   (6.1, 0.1) ;
\node (ld) at (6,0.75) {$\begin{matrix}\sigma = 1,\\ \tau =1 \end{matrix}$};


\draw[->] (-1,0.2) to [distance=0.25cm](-0.2,0.1);
\draw[->] (-0.2,-0.3)to [distance=-0.25cm] (-1,-0.4) ;
\draw[->] (-0.2,-0.5)to [distance=-0.25cm] (-1,-0.6) ;
\draw[->] (0.2,0.1) to [distance=0.5cm](1.8,0.1);
\draw[->] (1.8,-0.3) to [distance=-0.5cm](0.2,-0.3);
\draw[->] (1.8,-0.5) to [distance=-0.5cm](0.2,-0.5);
\node (xa) at (1,0.2) {\small{$X=1$}};
\node (y1b) at (1,-0.4) {\small{$Y_1=\frac{1}{q^2}-1$}};
\node (y1b) at (1,-0.9) {\small{$Y=-2$}};

\draw[->] (2.2,0.1) to [distance=0.5cm](3.8,0.1);
\node (xb) at (3,0.2) {\small{$X=1$}};
\draw[->] (3.8,-0.3) to [distance=-0.5cm](2.2,-0.3);
\draw[->] (3.8,-0.5) to [distance=-0.5cm](2.2,-0.5);
\node (y1b) at (3,-0.4) {\small{$Y_1=\frac{1}{q}-1$}};
\node (y1b) at (3,-0.9) {\small{$Y=-1$}};

\draw[->] (4.2,0.1) to [distance=0.5cm](5.8,0.1);
\node (xc) at (5,0.2) {\small{$X=1$}};
\draw[->] (5.8,-0.3) to [distance=-0.5cm](4.2,-0.3);
\draw[->] (5.8,-0.5) to [distance=-0.5cm](4.2,-0.5);
\node (y1c) at (5,-0.4) {\small{$Y_1=c$}};
\node (y1c) at (5,-0.9) {\small{$Y=d$}};

\draw[->] (6.2,0.1) to [distance=0.25cm](7,0.1);
\draw[->] (7,-0.3) to [distance=-0.25cm](6.2,-0.3);
\draw[->] (7,-0.5) to [distance=-0.25cm](6.2,-0.5);

\end{tikzpicture}

\end{center}

\begin{remark}\label{Aq-indec-no-D-str}
If $a \neq 0$, then for $J = B_{\omega}', J' = \emptyset $, the module
$V_q(\omega ,J, J')$ cannot be extended to a $D$-module. 
\end{remark}

\noindent
(3)  $J= \{ \mathfrak{m}_{0} \}, J'=\emptyset$, $\mathfrak{n}_0 =-\infty, \mathfrak{n}_1 
			=\mathfrak{m}_{0} $

In this case again, we have $V_q( \omega , J,J') = \oplus_{i \leq 0} \K v_i$.  
Set $\K v_i = K_{\mathfrak{m}_{i}}$ 
for $i\leq 0$. 
Here, 
\begin{align*}
X(v_i) &= \begin{cases} (q^i-1)v_{i+1} &\textit{for } i\leq -1,\\
		0 &\textit{for }i=0
		\end{cases};
&Y_1(v_i) = v_{i-1};\\
\sigma (v_i) &= q^{i-1}v_i; &\tau (v_i) = (a+i)v_i \quad \forall i\leq 0.
\end{align*}
Since $\sigma -1 \neq 0$, we can define action of $Y$ on on $V_q( \omega , J,J')$
by letting $Y = Y_1 \left( \dfrac{\tau -1}{\sigma -1} \right)$.
This gives an irreducible  $D$ weight module denoted by $V_q( \omega , J,J',a)$.
After a suitable change of basis, the picture here is

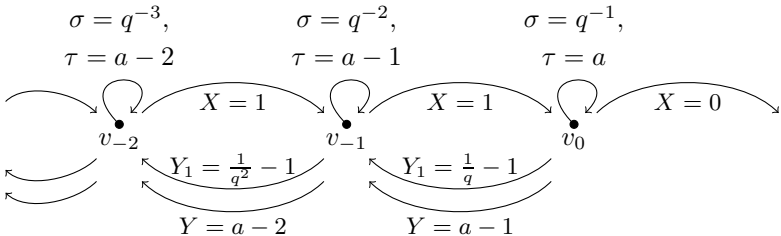
\begin{center}
\begin{tikzpicture}[scale=1.5]
\coordinate [label=below:$v_{-2}$] (A) at (0,0);
\coordinate [label=below:$v_{-1}$] (B) at (2,0);
\coordinate [label=below:$v_0$] (C) at (4,0);

\filldraw[black](A) circle (1pt);
\filldraw[black](B) circle (1pt);
\filldraw[black](C) circle (1pt);

\draw [->] (A)  .. controls  (-0.5,0.5) and (0.5,0.5)  ..   (0.1, 0.1) ;
\node (la) at (0,0.75) {$\begin{matrix}\sigma = q^{-3},\\ \tau =a-2 \end{matrix}$};

\draw [->] (B)  .. controls  (1.5,0.5) and (2.5,0.5)  ..   (2.1, 0.1) ;
\node (lb) at (2,0.75) {$\begin{matrix}\sigma = q^{-2},\\ \tau =a-1 \end{matrix}$};

\draw [->] (C)  .. controls  (3.5,0.5) and (4.5,0.5)  ..   (4.1, 0.1) ;
\node (lc) at (4,0.75) {$\begin{matrix}\sigma = q^{-1},\\ \tau =a \end{matrix}$};

\draw[->] (-1,0.2) to [distance=0.25cm](-0.2,0.1);
\draw[->] (-0.2,-0.3)to [distance=-0.25cm] (-1,-0.4) ;
\draw[->] (-0.2,-0.5)to [distance=-0.25cm] (-1,-0.6) ;
\draw[->] (0.2,0.1) to [distance=0.5cm](1.8,0.1);
\draw[->] (1.8,-0.3) to [distance=-0.5cm](0.2,-0.3);
\draw[->] (1.8,-0.5) to [distance=-0.5cm](0.2,-0.5);
\node (xa) at (1,0.2) {\small{$X=1$}};
\node (y1b) at (1,-0.4) {\small{$Y_1=\frac{1}{q^2}-1$}};
\node (y1b) at (1,-0.9) {\small{$Y=a-2$}};

\draw[->] (2.2,0.1) to [distance=0.5cm](3.8,0.1);
\node (xb) at (3,0.2) {\small{$X=1$}};
\draw[->] (3.8,-0.3) to [distance=-0.5cm](2.2,-0.3);
\draw[->] (3.8,-0.5) to [distance=-0.5cm](2.2,-0.5);
\node (y1b) at (3,-0.4) {\small{$Y_1=\frac{1}{q}-1$}};
\node (y1b) at (3,-0.9) {\small{$Y=a-1$}};

\draw[->] (4.2,0.1) to [distance=0.5cm](5.8,0.1);
\node (xc) at (5,0.2) {\small{$X=0$}};

\end{tikzpicture}

\end{center}

\begin{remark}\label{rem:VqJJ'3}
When characteristic of $\K$ is $0$, and $a\in \K \setminus \Z$ (respectively, $a\in \K \setminus \Z_p$
when characteristic of $\K$ is $p>0$)
then the resulting module $V_q(\omega ,b,a)$
is an irreducible $A_1$-module. If $a\in \Z$ (respectively, $a\in \Z_p$),
then the resulting module $V_q(\omega, b,a)$ is an indecomposable $A_1$-module. 
\end{remark}

\noindent
(4)  
$J = \{ \mathfrak{m}_{1} \}, J' = \emptyset$, $\mathfrak{n}_0 =\mathfrak{m}_{0}, \mathfrak{n}_1 			                       =\infty$
In this case  we have $V_q(\omega ,J,J') 
= \oplus_{i \geq 1} \K v_i$.  Set $\K v_i = K_{\mathfrak{m}_{i}}$ 
for $i\geq 1$. 
Here, 
\begin{align*}
X(v_i) &= (q^i-1)v_{i+1};
&Y_1(v_i) = 
 \begin{cases}  v_{i-1} &\textit{for } i\geq 2,\\
		0 &\textit{for }i=1
		\end{cases};
\\
\sigma (v_i) &= q^{i}v_i; &\tau (v_i) = (a+i)v_i \quad \forall i\geq 1.
\end{align*}
Since $\sigma -1 \neq 0$ on $v_i, i\geq 2$, we can define $Y$ on on $V_q( \omega , J,J')$
by letting $Y = Y_1 \left( \dfrac{\tau -1}{\sigma -1} \right)$ on $v_i, i\geq 2$ and $Y(v_1)=0$.
This gives an irreducible  $D$  module we denote by $V_q( \omega , J,J',a)$.
 Here the picture is:

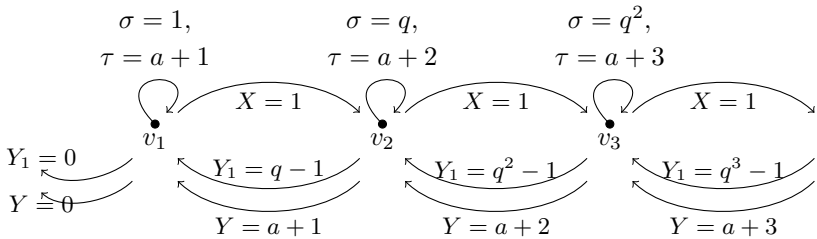
\begin{center}
\begin{tikzpicture}[scale=1.5]
\coordinate [label=below:$v_{1}$] (A) at (0,0);
\coordinate [label=below:$v_{2}$] (B) at (2,0);
\coordinate [label=below:$v_3$] (C) at (4,0);

\filldraw[black](A) circle (1pt);
\filldraw[black](B) circle (1pt);
\filldraw[black](C) circle (1pt);

\draw [->] (A)  .. controls  (-0.5,0.5) and (0.5,0.5)  ..   (0.1, 0.1) ;
\node (la) at (0,0.75) {$\begin{matrix}\sigma = 1,\\ \tau =a+1 \end{matrix}$};

\draw [->] (B)  .. controls  (1.5,0.5) and (2.5,0.5)  ..   (2.1, 0.1) ;
\node (lb) at (2,0.75) {$\begin{matrix}\sigma = q,\\ \tau =a+2 \end{matrix}$};

\draw [->] (C)  .. controls  (3.5,0.5) and (4.5,0.5)  ..   (4.1, 0.1) ;
\node (lc) at (4,0.75) {$\begin{matrix}\sigma = q^{2},\\ \tau =a+3 \end{matrix}$};

\draw[->] (-0.2,-0.3)to [distance=-0.25cm] (-1,-0.4) ;
\draw[->] (-0.2,-0.5)to [distance=-0.25cm] (-1,-0.6) ;
\node (y1a) at (-1,-0.3) {\small{$Y_1=0$}};
\node (y1a) at (-1,-0.7) {\small{$Y=0$}};

\draw[->] (0.2,0.1) to [distance=0.5cm](1.8,0.1);
\draw[->] (1.8,-0.3) to [distance=-0.5cm](0.2,-0.3);
\draw[->] (1.8,-0.5) to [distance=-0.5cm](0.2,-0.5);
\node (xa) at (1,0.2) {\small{$X=1$}};
\node (y1b) at (1,-0.4) {\small{$Y_1=q-1$}};
\node (y1b) at (1,-0.9) {\small{$Y=a+1$}};

\draw[->] (2.2,0.1) to [distance=0.5cm](3.8,0.1);
\node (xb) at (3,0.2) {\small{$X=1$}};
\draw[->] (3.8,-0.3) to [distance=-0.5cm](2.2,-0.3);
\draw[->] (3.8,-0.5) to [distance=-0.5cm](2.2,-0.5);
\node (y1b) at (3,-0.4) {\small{$Y_1=q^2-1$}};
\node (y1b) at (3,-0.9) {\small{$Y=a+2$}};

\draw[->] (4.2,0.1) to [distance=0.5cm](5.8,0.1);
\node (xc) at (5,0.2) {\small{$X=1$}};
\draw[->] (5.8,-0.3) to [distance=-0.5cm](4.2,-0.3);
\draw[->] (5.8,-0.5) to [distance=-0.5cm](4.2,-0.5);
\node (y1c) at (5,-0.4) {\small{$Y_1=q^3-1$}};
\node (y1c) at (5,-0.9) {\small{$Y=a+3$}};

\end{tikzpicture}

\end{center}

\begin{remark}\label{rem:VqJJ'4}
When characteristic of $\K$ is $0$, and $a\in \K \setminus \Z$ (respectively, $a\in \K \setminus \Z_p$
when characteristic of $\K$ is $p>0$)
then the resulting module $V_q(\omega ,b,a)$
is an irreducible $A_1$-module. If $a\in \Z$ (respectively, $a\in \Z_p$),
then the resulting module $V_q(\omega, b,a)$ is an indecomposable $A_1$-module. 
\end{remark}

\noindent
(5)  $J=\emptyset, J' = \emptyset$, $\mathfrak{n}_0 =\mathfrak{m}_{0}, \mathfrak{n}_1 =\infty$.
The resulting module is the same as the one above in case (4).

\noindent
(6)  $J=\emptyset, J' = \emptyset$, $\mathfrak{n}_0 =-\infty, \mathfrak{n}_1 =\infty$.
The resulting module is the same as the one above in case (2).

\subsubsection{The module $V_q(\omega ,J, J')$ when $q$ is a root of unit of $1$.
}\label{qpVJJ'}
Note that since $|\omega | =\infty$ and $q$ is a root of $1$, this case arises when characteristic of
$\K$ is $0$. 
Since $D$ contains $A_1(\K)$, we can extend only those $A_q$-modules which are infinite dimensional. In other words,  we consider only those $J \subset B'_{\omega}$ for which $\mathfrak{n}_0 = -\infty$ or $\mathfrak{n}_1 = \infty$.
We also have
 $b=q^k$ for some $k\in \Z$,  $|B_{\omega}| = \infty$ and 
$B'_{\omega}=B_{\omega}$. 
Set $J' \subset J$ not containing the maximal element of $J$ and 
we consider the indecomposable $A_q$-module $V_q(\omega ,J ,J')$.  We point out that
$V_q(\omega, J,J')$ is not irreducible $A_q$-module for any choice of $J'$.

Not every  $V_q(\omega ,J ,J')$ can be extended to a $D$ module.  
Using the condition $YX = \tau$ we see the following cases for any $v\neq 0, v\in V_{\mathfrak{m}}$:

\noindent
\textbf{Case 1:}  $X (v) = c \alpha (v) \neq 0$.  Then, set $Y(\alpha (v)) = \frac{\tau (v)}{c}$. 

\noindent
\textbf{Case 2:} For $\mathfrak{m} \in J\setminus J'$, 
$X(v) =0$,
$Y_1 (\alpha (v)) = v$ and
$\tau (v) =0$.  In this case set $Y(\alpha (v)) =dv$ for any $d\in \K$.

\noindent
\textbf{Case 3:}  For $\mathfrak{m} \in J\setminus J'$, 
$X(v) =0$,
$Y_1 (\alpha (v)) = v$ and
$\tau (v) \neq 0 $. In this case this $A_q$-module cannot be extended to a $D$-module
since we require $YX (v) = \tau (v)$. In particular, since the characteristic of $\K$ is $0$,
we have that $\tau=0$ can happen for at most one weight  $\mathfrak{m}$. Thus, 
if $\mathfrak{m}_1, \mathfrak{m}_2 \in J\setminus {J}'$ are such that
$X(v_1)=X(v_2)=0$ for $v_1,v_2\neq 0, v_1\in \mathfrak{m}_1, v_2\in \mathfrak{m}_1$,
then this $A_q$-module cannot be extended to a $D$-module.

\begin{remark}\label{Aqp-indec-no-D-str}
In case there is a $D$-module structure, then that structure is indecomposable $D$-module since it was indecomposable as an $A_q$-module to begin with.  This module is irreducible as a $D$ module
if and only if $X \neq 0$ for any $\mathfrak{m} \neq
\mathfrak{n}_1$ and $Y, Y_1$ are not simultaneously zero unless $\alpha^{-1} (\mathfrak{m}) = \mathfrak{n}_0$.  
Further, if 
$X \neq 0$ for any $\mathfrak{m} \neq
\mathfrak{n}_1$ and $Y(v)=Y_1(w)=0$ for nonzero $v,w, v\neq w, v,w \notin V_{\alpha(\mathfrak{n}_0)}$ 
then we have an irreducible $D$-module which is not irreducible as an $A_1$ or an $A_q$-module.

\end{remark}

\subsubsection{The module $V_q(\omega ,f,b)$ for $b\notin \{ q^i \}_{i\in \Z}$.}\label{qVomegaf}
This family of $A_q$-decomposables arise when $|\omega|<\infty$ and $B_{\omega}=\emptyset$. 
In particular, the characteristic of $\K =p>0$, and $q$ is a root of $1$
Since $B_{\omega}=\emptyset$, we have $\omega = \alpha^i (\mathfrak{m})$ where
$\mathfrak{m} = (\tau-a, \sigma -b)$ where $a,b\in \K, b\neq 0, b\notin \{ q^i \}_{i\in \Z}$. 
Suppose $|\omega|=r$ (therefore, $q^r=1$ and $p/r$).
For any $f\in \K^*$, 
let $V_q(\omega ,f,b) = \oplus_{1\leq i\leq r} \K v_i$ be an $r$-dimensional vector space and 
give it $A_q$-module structure 
as described in section \ref{general-set-up}, with $A=A_q$, $T=Y_1$, and 
$V_q(\omega, f,b) = V(\omega, f)$
setting 
\[
\sigma (v_i) = q^{i-1}bv_i; \quad
X(v_i) = \begin{cases} 
            (q^ib-1)v_{i+1} &\textit{ if }i<r,\\
             f (b-1)v_1 &\textit{ if } i=r;
		\end{cases}
\]
\[
\tau (v_i) = (a+i-1)v_i; \quad
Y_1(v_i) = \begin{cases} 
             v_{i-1} &\textit{ if } i>1,\\
             \frac{1}{f}v_r &\textit{ if } i=1.
		\end{cases}
\]
Since action of $X$ is an isomorphism, we extend the action of $A_q$ to an action of $D$
by using $YX = \tau$ and setting 
\[
Y(v_{i+1}) = \begin{cases}
			\dfrac{a+i-1}{q^{i}b-1} v_i &\textit{ if } i >1;\\
			\dfrac{a+r-1}{f(b-1)} v_r &\textit{ if } i =1.
		\end{cases}
\]
We denote the resulting irreducible $D$-module by $V_q(\omega ,f,b, a)$.
Note that the resulting $A_1$-module is also irreducible since $X$ is invertible. 
With a change of basis we see the picture as given:

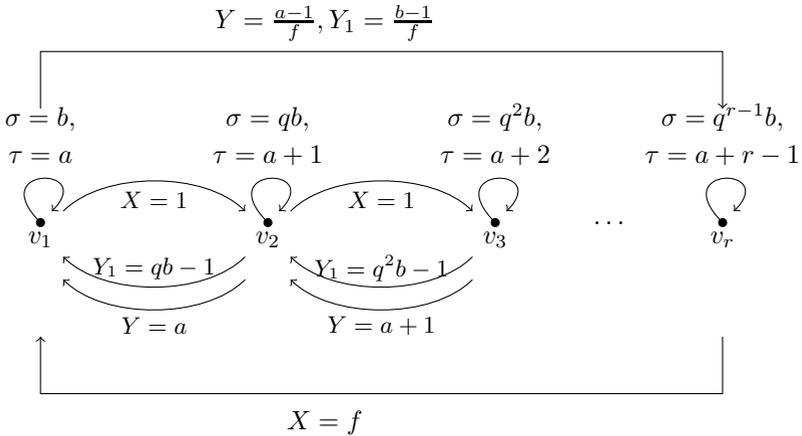
\begin{center}
\begin{tikzpicture}[scale=1.5]
\coordinate [label=below:$v_{1}$] (A) at (0,0);
\coordinate [label=below:$v_{2}$] (B) at (2,0);
\coordinate [label=below:$v_3$] (C) at (4,0);
\node (la) at (5,0) {$\ldots$};
\coordinate [label=below:$v_{r}$] (D) at (6,0);

\filldraw[black](A) circle (1pt);
\filldraw[black](B) circle (1pt);
\filldraw[black](C) circle (1pt);
\filldraw[black](D) circle (1pt);

\draw[->] (0.2,0.1) to [distance=0.5cm](1.8,0.1);
\draw[->] (1.8,-0.3) to [distance=-0.5cm](0.2,-0.3);
\draw[->] (1.8,-0.5) to [distance=-0.5cm](0.2,-0.5);
\node (xa) at (1,0.2) {\small{$X=1$}};
\node (y1b) at (1,-0.4) {\small{$Y_1=qb-1$}};
\node (yb) at (1,-0.9) {\small{$Y=a$}};

\draw[->] (2.2,0.1) to [distance=0.5cm](3.8,0.1);
\draw[->] (3.8,-0.3) to [distance=-0.5cm](2.2,-0.3);
\draw[->] (3.8,-0.5) to [distance=-0.5cm](2.2,-0.5);
\node (xb) at (3,0.2) {\small{$X=1$}};
\node (yc) at (3,-0.4) {\small{$Y_1=q^2b-1$}};
\node (y1c) at (3,-0.9) {\small{$Y=a+1$}};


\draw[->](0,1)--(0,1.5)--(6,1.5)--(6,1);
\node (ly1r) at (2.5,1.75){$Y=\frac{a-1}{f}, Y_1=\frac{b-1}{f}$};

\draw[<-](0,-1)--(0,-1.5)--(6,-1.5)--(6,-1);
\node (ly1r) at (2.5,-1.75){$X=f$};

\draw [->] (A)  .. controls  (-0.5,0.5) and (0.5,0.5)  ..   (0.1, 0.1) ;
\node (la) at (0,0.75) {$\begin{matrix}\sigma = b,\\ \tau =a \end{matrix}$};

\draw [->] (B)  .. controls  (1.5,0.5) and (2.5,0.5)  ..   (2.1, 0.1) ;
\node (lb) at (2,0.75) {$\begin{matrix}\sigma = qb,\\ \tau =a+1 \end{matrix}$};

\draw [->] (C)  .. controls  (3.5,0.5) and (4.5,0.5)  ..   (4.1, 0.1) ;
\node (lc) at (4,0.75) {$\begin{matrix}\sigma = q^{2}b,\\ \tau =a+2 \end{matrix}$};

\draw [->] (D)  .. controls  (5.5,0.5) and (6.5,0.5)  ..   (6.1, 0.1) ;
\node (ld) at (6,0.75) {$\begin{matrix}\sigma = q^{r-1}b,\\ \tau =a+r-1 \end{matrix}$};


\end{tikzpicture}

\end{center}

\subsubsection{The module $V_q(\omega ,j, w)$.}\label{qVomegajw}

This family of $A_q$-decomposables arise when $|\omega|<\infty$ and $B_{\omega}\neq \emptyset$
and this is Family 1 as described in section \ref{general-set-up}. 
Let $|B_{\omega}| = m>0$ (that is, $b \in \{ q^i \}_{i\in \Z}$). 
The one-to-one correspondence $\Z_m \to B_m$
and $j(\mathfrak{m}) \in \Z_m$ are as described in 
section \ref{general-set-up}.  Let $x,y$ be
two noncommuting variables.  

Fix $j\in \Z_m$ and $w=z_1z_2\cdots z_n$ be a word of length 
$n\geq 1$ where each $z_i \in \{ x, y\}$.  Let $e_0,e_1,\ldots, e_n$ be $n+1$ symbols. For each $\mathfrak{m}\in \omega$, let $V_{\mathfrak{m}}$ be a vector space over $\K_{\mathfrak{m}}$ with basis
$\{ (\mathfrak{m}, e_k) \mid k+j = j(\mathfrak{m}) \in \Z_m \}$. Put 
$V_q(\omega ,j, w) = \oplus_{\mathfrak{m} \in \omega} V_{\mathfrak{m}}$ and it has an $A_q$-module 
structure given by  
\[
X (\mathfrak{m},e_k) = \begin{cases} t_{\mathfrak{m}} (\alpha (\mathfrak{m}), e_k) &\textit{ if } \mathfrak{m} \notin B,\\
							(\alpha (\mathfrak{m}), e_{k+1}) &\textit{ if } \mathfrak{m} \in 
                             							B_{\omega}
										\textit{ and } z_{k+1}=x,\\ 
							0 &\textit{ otherwise};  
				\end{cases} 
\]
\[
Y_1 (\mathfrak{m},e_k) = \begin{cases}  (\alpha^{-1} (\mathfrak{m}), e_k) &\textit{ if } \mathfrak{m} 
										\notin B,\\
							(\alpha^{-1} (\mathfrak{m}), e_{k-1}) &\textit{ if } \mathfrak{m} \in B_{\omega} 
										\textit{ and } z_{k}=y,\\ 
							0 &\textit{ otherwise}.  
				\end{cases}
\]
The module $V_q(\omega ,j, w) $ is irreducible as an $A_q$-module when $w$ is the empty word. 
Not every  $V_q(\omega ,j,w)$ can be extended to a $D$ module.  
Using the condition $YX = \tau$ we see the following cases for any $v\neq 0, v\in V_{\mathfrak{m}}$:

\noindent
\textbf{Case 1:}  $X (v) = c \alpha (v) \neq 0$.  Then, set $Y(\alpha (v)) = \frac{\tau (v)}{c}$. 

\noindent
\textbf{Case 2:} For $\mathfrak{m} \in B_{\omega}$, 
$X(v) =0$,
$Y_1 (\alpha (v)) = v$ and
$\tau (v) =0$.  In this case set $Y(\alpha (v)) =dv$ for any $d\in \K$.

\noindent
\textbf{Case 3:}  For $\mathfrak{m} \in B_{\omega}$, 
$X(v) =0$,
$Y_1 (\alpha (v)) = v$ and
$\tau (v) \neq 0 $. In this case this $A_q$-module cannot be extended to a $D$-module
since we require $YX (v) = \tau (v)$.

In cases 1 and 2, we obtain an indecomposable $D$-module since it was indecomposable as an $A_q$-module to begin with.  
\begin{remark}\label{Aqomegajw-indec-no-D-str}
\begin{itemize}
\item
This module is irreducible as a $D$ module
if and only if $X \neq 0$ at any vector other than $(\mathfrak{m}_1, e_n)$ and $Y, Y_1$ are not simultaneously zero other than at $(\mathfrak{m}_0, e_0)$. 
If in addition, 
if $Y(v)=0$ and $Y_1(w)=0$ for some $v\neq w$
then we have an irreducible $D$-module which is not irreducible as an $A_1$ or an $A_q$-module. 
\item
Another point to note, as we have seen in Case 3 above, not every $A_q$-module has a $D$-module structure.
\item In Case 2 above, we could have $Y(\alpha (v)) =dw$ for some other $w$ with $X(w)=0, \tau (w)=0$
	thereby giving us yet another extension of an $A_q$-module to a $D$-module.
\end{itemize}
\end{remark}

\subsubsection{The module $V_q(\omega ,w,f)$.}\label{qVomegawf}
This family of $A_q$-decomposables arise when $|\omega|<\infty$ and $B_{\omega}\neq \emptyset$
and this is Family 2 as described in section \ref{general-set-up}.
Let $w$ be a word of length $n$ where $n$ is a multiple of $m$
and let $f\in \K^*$. Consider $n$ elements $e_{k}$ where
$k=1,2,\ldots ,n$. For $\mathfrak{m}\in \omega$, let $V_{\mathfrak{m}}$ be a $\K_{\mathfrak{m}}$ vector
space with basis $\{ e_{k} \mid k \equiv j (\textit{mod }m) \}$. The vector space 
$V_q(\omega ,w, f) = \oplus_{\mathfrak{m}\in \omega} V_{\mathfrak{m}}$ has an $A$-module structure given by
\[
X(\mathfrak{m},e_{k}) = \begin{cases} 
                                      t_{\mathfrak{m}}( \alpha( \mathfrak{m}), e_{k} ) 
						&\textit{ if }\mathfrak{m} \textit{ is not a break;}\\
				 ( \alpha( \mathfrak{m}), e_{k+1} ) 
						&\textit{ if }\mathfrak{m} \in B_{\omega}, k \neq n, z_{k+1} =x;\\
				f( \alpha( \mathfrak{m}), e_{1} ) 
						&\textit{ if }\mathfrak{m} \in B_{\omega}, k = n, z_{1} =x;\\
				0 &\textit{ otherwise;} 
				\end{cases}
\]
\[
Y_1(\mathfrak{m},e_{k}) = \begin{cases} 
                                      ( \alpha ^{-1}( \mathfrak{m}), e_{k} ) 
						&\textit{ if }\alpha^{-1}(\mathfrak{m}) \textit{ is not a break;}\\
				 ( \alpha ^{-1}( \mathfrak{m}), e_{k-1} ) 
						&\textit{ if }\alpha^{-1}((\mathfrak{m} )\in B_{\omega}, k \neq 1, z_{k} =y;\\
				f( \alpha ^{-1}( \mathfrak{m}), e_{n} ) 
						&\textit{ if }\alpha^{-1}(\mathfrak{m}) \in B, k = 1, z_{1} =y, s=1;\\
				0 &\textit{ otherwise;} 
				\end{cases}
\]
The module $V_q(\omega, w,f)$ is irreducible if $w =x^{m}$ or $w=y^{m}$. 
The extension of this action to an action of $D$ follows the same pattern as in the previous case 
 $V_q(\omega ,j, w)$, section \ref{qVomegajw}. That is, 
 
\noindent
\textbf{Case 1:}  $X (v) = c \alpha (v) \neq 0$.  Then, set $Y(\alpha (v)) = \frac{\tau (v)}{c}$. 

\noindent
\textbf{Case 2:} For $\mathfrak{m} \in B_{\omega}$, 
$X(v) =0$,
$Y_1 (\alpha (v)) = v$ and
$\tau (v) =0$.  In this case set $Y(\alpha (v)) =dv$ for any $d\in \K$.

\noindent
\textbf{Case 3:}  For $\mathfrak{m} \in B_{\omega}$, 
$X(v) =0$,
$Y_1 (\alpha (v)) = v$ and
$\tau (v) \neq 0 $. In this case this $A_q$-module cannot be extended to a $D$-module
since we require $YX (v) = \tau (v)$. The extended $D$-module is also indecomposable. 

\begin{remark}\label{Aqomegajw-indec-no-D-str}
\begin{itemize}
\item
This module is irreducible as a $D$ module
if and only if $X =0 $ at at most one basis vector and $Y, Y_1$ 
are not simultaneously zero. 
\item
Not every $A_q$-module has a $D$-module structure.
\item In Case 2 above, we could have $Y(\alpha (v)) =dw$ for some other $w$ with $X(w)=0, \tau (w)=0$
	thereby giving us yet another extension of an $A_q$-module to a $D$-module.
\end{itemize}
\end{remark}
\subsection{The list of $A_1$-indecomposables.} \label{DGO-indecomposables-of-A1-char0}

Now we go through the list of GWA-indecomposables as
described in \cite{DGO} and understand them as $A_1$-weight modules.

\subsubsection{The module $V_1(\omega, a)$}\label{Vomega}
Heret $|\omega|=\infty$ and $B_{\omega}= \emptyset$. Hence, $a\in \K \setminus \Z$ if the characteristic
of $\K$ is $0$, and $a\in \K \setminus \Z_p$ if the characteristic of $\K$ is $p>0$.
Here, $V_1(\omega, a) = \oplus_{i\in \Z} \K v_i$ where $\{ v_i \}_{i \in \Z}$ is a basis,
and the action is given by
\[
X(v_i)= (a+i)v_{i+1},\quad
Y(v_i) = v_{i-1},\quad
f(\tau, \sigma) (v_i) = f (a+i, q^ib) v_i, \quad \textit{for } f\in \K [\tau, \sigma, \sigma^{-1}].
\]
For any $b\in \K^*$,  we may extend the action of $A_1$ on $V_1(\omega,a)$ to an action of $D$ by setting
$
Y_1 (v_i) = \dfrac{(q^ib-1)}{(a+i-1)} v_{i-1}.
$
The resulting module is denoted $V_1(\omega,a,b)$ and is an irreducible $D$-module. 
By a suitable change of basis the picture for $V_1(\omega ,a,b)$ is identical to the picture of
$V_q(\omega ,b,a)$ drawn  in  case \ref{qVomega}.

\subsubsection{The module $V_1(\omega ,J, J')$ when characteristic of $\K$ is $0$}\label{VJJ'}
Suppose that $|\omega|= \infty$ and $B_{\omega} \neq \emptyset$.
That is, $\omega = \{ \alpha ^i (\tau-a, \sigma -b) \}_{i\in \Z}$.
Since $B=\{ (\tau , \sigma -c) \mid c\in \K \}$, we 
have $B_{\omega} \neq \emptyset$ if and only if $a\in \Z$. Without loss of generality,
$a =0$. 
That is, 
$
\omega = \{  \mathfrak{m}_i = (\tau  -i, \sigma - q^{i}b) \}_{i\in \Z}$,
  and 
$B_{\omega} = \left\{  \mathfrak{m}_{0} = \left( \tau , \sigma  - b \right) \right\}$.
Set $B_{\omega}' = \{ \mathfrak{m}_{0}, \mathfrak{m}_{1} \}$.  Thus, 
we have the following options:

\noindent
(1)
  $J = B_{\omega}', J' = \{ \mathfrak{m}_{0} \}$, $\mathfrak{n}_0 =-\infty, \mathfrak{n}_1 =\infty$

In this case, we have $V_1( \omega , J,J') = \oplus_{i=-\infty}^{\infty} \K v_i$.  Set $\K v_i = K_{\mathfrak{m}_{i}}$ 
for $i\in \Z$. 
\[
X(v_i) = \begin{cases}
	v_1 &\textit{if } i = 0\\
	 i v_{i+1}&\textit{if }i \neq 0;
	\end{cases}
\quad 
Y(v_i) = \begin{cases}
	0 &\textit{if }i = 1\\
	v_{i-1}&\textit{if } i\neq 1
	\end{cases};
\]
\[
\sigma (v_i) = q^{i}bv_i, \quad \tau (v_i) = iv_i \quad \forall i\in \Z.
\]
Using the fact that $Y_1X = q\sigma -1$, for a fixed $b\in \K^*$ we define an
action of $Y_1$ on $V_1( \omega , J,J')$ as follows:
$
Y_1(v_{i+1}) = \begin{cases} \dfrac{(q^{1+i}b-1)}{i} v_i & \forall i\neq 0\\
				 	 (qb -1)v_0.  &i=0.
		\end{cases}
$
The resulting $D$ module, denoted by $V_1( \omega , J,J',b)$, is irreducible if $b\neq \frac{1}{q}$. 
 If $b = q^i$ for some $i\in \Z, i \neq -1$,
then this module is irreducible as a $D$ module even though it is not irreducible as an $A_1$ module or an $A_q$ module.
The picture for this module is identical to the one drawn in subsection \ref{qVomega} for $a=0$.

\noindent
(2)  $J = B_{\omega}', J' = \emptyset $, $\mathfrak{n}_0 =-\infty, \mathfrak{n}_1 =\infty$.

In this case again, we have $V_1(\omega,J,J') = \oplus_{i=-\infty}^{\infty} \K v_i$.  Set $\K v_i = K_{\mathfrak{m}_{i}}$ 
for $i\in \Z$. 
Here, 
\[
X(v_i) = iv_{i+1};
\quad 
Y(v_i) = v_{i-1};
\quad
\sigma (v_i) = q^{i}bv_i, \quad \tau (v_i) = iv_i \quad \forall i\in \Z.
\]
Using the fact that $Y_1X = q\sigma -1$, 
we now define action of $Y_1$ on $V( \omega , J,J')$ by
\[
Y_1(v_i) = \begin{cases}
			cv_0 &\textit{if }i =1, \textit{ for some }c\in \K,\\
			\left( \dfrac{q^{i}b -1}{i-1} \right) v_{i-1} &\textit{if } i\neq 1.	
		\end{cases}
\]
Requiring that $XY_1 (v_1) = (\sigma -1)(v_1)$ gives us $b = q^{-1}$.
Thus, we have the following:
\[
\sigma (v_i) = q^{i-1}v_i; \quad 
Y_1(v_i) = \begin{cases}
			cv_0 &\textit{if }i =1, \textit{ for some }c\in \K,\\
			\left( \frac{q^{i-1}-1}{i-1} \right) v_{i-1} &\textit{if } i\neq 1.	
		\end{cases}
\]
For any fixed $c\in \K$, we have a $D$-module denoted by $V_1(\omega ,J,J',c)$, 
which is indecomposable, but not irreducible. 
\begin{remark}\label{A1-indec-no-D-str}
 If  $b\neq q^{-1}$, then $V_1(\omega ,J,J')$ does not have a $D$-module structure. 
\end{remark}

We may generalize this construction to the following $D$-module denoted by $V_1(\omega ,J,J',c,d)$
for any $c,d\in \K$. 
Set $V_1(\omega ,J,J',c,d) = \oplus_{i=-\infty}^{\infty} \K v_i$.  
Here, 
\[
X(v_i) = i v_{i+1};
\quad 
\sigma (v_i) = q^{i-1}v_i; \quad \tau (v_i) = iv_i \quad \forall i\in \Z.
\]
\[
Y(v_i) = \begin{cases}
			cv_0 &\textit{if }i =1, \textit{ for some }c\in \K,\\
			 v_{i-1} &\textit{if } i\neq 1;	
		\end{cases}
\]
\[
Y_1(v_i) = \begin{cases}
			dv_0 &\textit{if }i =1, \textit{ for some }d\in \K,\\
			\left( \dfrac{q^{i-1}-1}{i-1} \right) v_{i-1} &\textit{if } i\neq 1.	
		\end{cases}
\]
The $D$-module $V_1(\omega ,J,J',c,d)$ is indecomposable. Its picture is identical to that of 
\ref{qVJJ'}, case (2).

\noindent
(3)  $J= \{ \mathfrak{m}_{0} \}, J'=\emptyset$, $\mathfrak{n}_0 =-\infty, \mathfrak{n}_1 
			=\mathfrak{m}_{0} $

In this case again, we have $V_1(\omega,J,J') = \oplus_{i \leq 0} \K v_i$.  
Set $\K v_i = K_{\mathfrak{m}_{i}}$ 
for $i\leq 0$. 
Here, $Y(v_i) = v_{i-1};
\quad
\sigma (v_i) = q^{i}bv_i, \quad \tau (v_i) = iv_i$ and  
\[
X(v_i) = \begin{cases} iv_{i+1} &\textit{for } i\leq -1\\
		0 &\textit{for }i=0
		\end{cases}  \quad \forall i\leq 0.
\]
Since $\tau -1 \neq 0$, we can define $Y_1$ on on $V_1( \omega , J,J')$
by letting $Y_1 = Y \left( \dfrac{\sigma -1}{\tau -1} \right)$.
This gives a module structure of $D$ on on $V_1( \omega , J,J')$, and it is irreducible.
The picture here is identical to that of \ref{qVJJ'}, case (3).

\noindent
(4)  
$J = \{ \mathfrak{m}_{1} \}, J' = \emptyset$, $\mathfrak{n}_0 =\mathfrak{m}_{0}, \mathfrak{n}_1 			                       =\infty$
In this case  we have $V_1( \omega , J,J') = \oplus_{i \geq 1} \K v_i$.  Set $\K v_i = K_{\mathfrak{m}_{i}}$ 
for $i\geq 1$. 
Here,  $\sigma (v_i) = q^{i-1}bv_i, \quad \tau (v_i) = iv_i $ and 
\[
X(v_i) =iv_{i+1}
\quad
Y(v_i) = 
 \begin{cases}  v_{i-1} &\textit{for } i\geq 2,\\
		0 &\textit{for }i=1
		\end{cases};
\quad \forall i\geq 1.
\]
Since $\tau -1 \neq 0$ on $v_i, i\geq 2$, we can define $Y_1$ on on $V_1( \omega , J,J')$
by letting $Y_1 = Y \left( \dfrac{\sigma -1}{\tau -1} \right)$ on $v_i, i\geq 2$ and $Y_1(v_1)=0$.
This gives a module structure of $D$ on on $V_1( \omega , J,J')$, and it is irreducible.
 Here the picture is identical to that of \ref{qVJJ'}, case (4).

\noindent
(5)  $J=\emptyset, J' = \emptyset$, $\mathfrak{n}_0 =\mathfrak{m}_{0}, \mathfrak{n}_1 =\infty$.
The resulting module is the same as the one above in case (4).

\noindent
(6)  $J=\emptyset, J' = \emptyset$, $\mathfrak{n}_0 =-\infty, \mathfrak{n}_1 =\infty$.
The resulting module is the same as the one above in case (2).

\subsubsection{The module $V_1(\omega ,J, J')$ when characteristic of $\K$ is nonzero}\label{VJJ'p}
Note that since $|\omega | =\infty$ and characteristic of $\K$ is $p>0$, 
this case arises when $q$ is not a root of $1$.
Since $D$ contains $A_1^q(\K)$, we can extend only those $A_1$-modules which are infinite dimensional. Thus, we consider only those $J \subset B'_{\omega}$ for which $\mathfrak{n}_0 = -\infty$ or $\mathfrak{n}_1 = \infty$.
We also have
 $a\in \Z_p$,  $|B_{\omega}| = \infty$ and 
$B'_{\omega}=B_{\omega}$. 
Set $J' \subset J$ not containing the maximal element of $J$ and 
we consider the indecomposable $A_1$-module $V_1(\omega ,J ,J')$.  We point out that
$V_1(\omega, J,J')$ is not irreducible $A_1$-module for any choice of $J'$.

Not every  $V_q(\omega ,J ,J')$ can be extended to a $D$ module.  
Using the condition $YX = \tau$ we see the following cases for any $v\neq 0, v\in V_{\mathfrak{m}}$:

\noindent
\textbf{Case 1:}  $X (v) = c \alpha (v) \neq 0$.  Then, set $Y_1(\alpha (v)) = \frac{(q\sigma -1) (v)}{c}$. 

\noindent
\textbf{Case 2:} For $\mathfrak{m} \in J\setminus J'$, 
$X(v) =0$,
$Y (\alpha (v)) = v$ and
$q\sigma (v) =v$.  In this case set $Y_1(\alpha (v)) =dv$ for any $d\in \K$.

\noindent
\textbf{Case 3:}  For $\mathfrak{m} \in J\setminus J'$, 
$X(v) =0$,
$Y (\alpha (v)) = v$ and
$q\sigma (v) \neq v $. In this case this $A_1$-module cannot be extended to a $D$-module
since we require $Y_1X (v) = (q\sigma -1) (v)$.
In particular, since $q$ is not a root of $1$,
we have that $q\sigma =1$ can happen for at most one weight  $\mathfrak{m}$. Thus, 
if $\mathfrak{m}_1, \mathfrak{m}_2 \in J\setminus {J}'$ are such that
$X(v_1)=X(v_2)=0$ for $v_1,v_2\neq 0, v_1\in \mathfrak{m}_1, v_2\in \mathfrak{m}_1$,
then this $A_1$-module cannot be extended to a $D$-module.

\begin{remark}\label{A1p-indec-no-D-str}
In case there is a $D$-module structure, then that structure is indecomposable $D$-module since it was indecomposable as an $A_1$-module to begin with.  This module is irreducible as a $D$ module
if and only if $X \neq 0$ for any $\mathfrak{m} \neq
\mathfrak{n}_1$ and $Y, Y_1$ are not simultaneously zero unless $\alpha^{-1} (\mathfrak{m}) = \mathfrak{n}_0$.  
Further, if 
$X \neq 0$ for any $\mathfrak{m} \neq
\mathfrak{n}_1$ and $Y(v)=Y_1(w)=0$ for nonzero $v,w, v\neq w, v,w \notin V_{\alpha(\mathfrak{n}_0)}$ 
then we have an irreducible $D$-module which is not irreducible as an $A_1$ or an $A_q$-module.

\end{remark}

\subsubsection{The module $V_1(\omega ,f,a)$ for $a\notin \Z_p$.}\label{1Vomegaf}
This family of $A_q$-decomposables arise when $|\omega|<\infty$ and $B_{\omega}=\emptyset$. 
In particular, the characteristic of $\K =p>0$, and $q$ is a root of $1$
Since $B_{\omega}=\emptyset$, we have $\omega = \alpha^i (\mathfrak{m})$ where
$\mathfrak{m} = (\tau-a, \sigma -b)$ where $a,b\in \K, b\neq 0, a\notin \Z_p$. 
Suppose $|\omega|=r$ (therefore, $q^r=1$ and $p/r$).
For any $f\in \K^*$, 
let $V_1(\omega ,f,a) = \oplus_{1\leq i\leq r} \K v_i$ be an $r$-dimensional vector space and 
give it $A_1$-module structure 
as described in section \ref{general-set-up}, with $A=A_1$, $T=Y$, and 
$V(\omega, f,a) = V(\omega, f)$
setting 
\[
\begin{matrix}
\sigma (v_i) = q^{i-1}bv_i;\\
\tau (v_i) = (a+i-1)v_i;
\end{matrix}
\quad
X(v_i) = \begin{cases} 
            (a+i-1)v_{i+1} &\textit{ if }i<r,\\
             f (a-1)v_1 &\textit{ if } i=r;
		\end{cases} 
\]
\[
Y(v_i) = \begin{cases} 
             v_{i-1} &\textit{ if } i>1,\\
             \frac{1}{f}v_r &\textit{ if } i=1.
		\end{cases}
\]
Since action of $X$ is an isomorphism, we extend the action of $A_1$ to an action of $D$
by using $Y_1X = q\sigma -1$ and setting 
\[
Y_1(v_{i+1}) = \begin{cases}
			\dfrac{q^{i}b-1}{a+i-1} v_i &\textit{ if } i >1;\\
			\dfrac{b-1}{f(a-1)} v_r &\textit{ if } i =1.
		\end{cases}
\]
We denote the resulting irreducible $D$-module by $V_1(\omega ,f,a, b)$.
With a change of basis we see the picture as given in section \ref{qVomegaf}

\subsubsection{The module $V_1(\omega ,j, w)$.}\label{1Vomegajw}
This family of $A_1$-decomposables arise when $|\omega|<\infty$ and $B_{\omega}\neq \emptyset$
and this is Family 1 as described in section \ref{general-set-up}. 
Let $|B_{\omega}| = m>0$ (that is, $a \in  \Z_p$). 
The one-to-one correspondence $\Z_m \to B_m$
and $j(\mathfrak{m}) \in \Z_m$ are as described in 
section \ref{general-set-up}.  Let $x,y$ be
two noncommuting variables.  

Fix $j\in \Z_m$ and $w=z_1z_2\cdots z_n$ be a word of length 
$n\geq 1$ where each $z_i \in \{ x, y\}$.  Let $e_0,e_1,\ldots, e_n$ be $n+1$ symbols. For each $\mathfrak{m}\in \omega$, let $V_{\mathfrak{m}}$ be a vector space over $\K_{\mathfrak{m}}$ with basis
$\{ (\mathfrak{m}, e_k) \mid k+j = j(\mathfrak{m}) \in \Z_m \}$. Put 
$V_1(\omega ,j, w) = \oplus_{\mathfrak{m} \in \omega} V_{\mathfrak{m}}$ and it has an $A_1$-module 
structure given by  
\[
X (\mathfrak{m},e_k) = \begin{cases} t_{\mathfrak{m}} (\alpha (\mathfrak{m}), e_k) &\textit{ if } \mathfrak{m} \notin B,\\
							(\alpha (\mathfrak{m}), e_{k+1}) &\textit{ if } \mathfrak{m} \in 
                             							B_{\omega}
										\textit{ and } z_{k+1}=x,\\ 
							0 &\textit{ otherwise};  
				\end{cases}
\]
\[
Y (\mathfrak{m},e_k) = \begin{cases}  (\alpha^{-1} (\mathfrak{m}), e_k) &\textit{ if } \mathfrak{m} 
										\notin B,\\
							(\alpha^{-1} (\mathfrak{m}), e_{k-1}) &\textit{ if } \mathfrak{m} \in B_{\omega} 
										\textit{ and } z_{k}=y,\\ 
							0 &\textit{ otherwise}.  
				\end{cases}
\]
The module $V_1(\omega ,j, w) $ is irreducible as an $A_1$-module when $w$ is the empty word. 
Not every  $V_1(\omega ,j,w)$ can be extended to a $D$ module.  
Using the condition $Y_1X = q\sigma -1$ we see the following cases for any $v\neq 0, v\in V_{\mathfrak{m}}$:

\noindent
\textbf{Case 1:}  $X (v) = c \alpha (v) \neq 0$.  Then, set $Y_1(\alpha (v)) = \frac{(q\sigma -1) (v)}{c}$. 

\noindent
\textbf{Case 2:} For $\mathfrak{m} \in B_{\omega}$, 
$X(v) =0$,
$Y (\alpha (v)) = v$ and
$q\sigma (v) =v$.  In this case set $Y_1(\alpha (v)) =dv$ for any $d\in \K$.

\noindent
\textbf{Case 3:}  For $\mathfrak{m} \in B_{\omega}$, 
$X(v) =0$,
$Y (\alpha (v)) = v$ and
$q\sigma (v) \neq v $. In this case this $A_1$-module cannot be extended to a $D$-module
since we require $Y_1X (v) = (q\sigma -1)(v)$.

In cases 1 and 2, we obtain an indecomposable $D$-module since it was indecomposable as an $A_1$-module to begin with.  
\begin{remark}\label{Aqomegajw-indec-no-D-str}
\begin{itemize}
\item
This module is irreducible as a $D$ module
if and only if $X \neq 0$ at any vector other than $(\mathfrak{m}_1, e_n)$ and $Y, Y_1$ are not simultaneously zero other than at $(\mathfrak{m}_0, e_0)$. 
If in addition, 
if $Y(v)=0$ and $Y_1(w)=0$ for some $v\neq w$
then we have an irreducible $D$-module which is not irreducible as an $A_1$ or an $A_q$-module. 
\item
Another point to note, as we have seen in Case 3 above, not every $A_1$-module has a $D$-module structure.
\item In Case 2 above, we could have $Y_1(\alpha (v)) =dw$ for some other $w$ with $X(w)=0, q\sigma (w)=w$
	thereby giving us yet another extension of an $A_1$-module to a $D$-module.
\end{itemize}
\end{remark}

\subsubsection{The module $V_1(\omega ,w,f)$.}\label{1Vomegawf}
This family of $A_1$-decomposables arise when $|\omega|<\infty$ and $B_{\omega}\neq \emptyset$
and this is Family 2 as described in section \ref{general-set-up}.
Let $w$ be a word of length $n$ where $n$ is a multiple of $m$
and let $f\in \K^*$. Consider $n$ elements $e_{k}$ where
$k=1,2,\ldots ,n$. For $\mathfrak{m}\in \omega$, let $V_{\mathfrak{m}}$ be a $\K_{\mathfrak{m}}$ vector
space with basis $\{ e_{k} \mid k \equiv j (\textit{mod }m) \}$. The vector space 
$V_1(\omega ,w, f) = \oplus_{\mathfrak{m}\in \omega} V_{\mathfrak{m}}$ has an $A$-module structure given by
\[
X(\mathfrak{m},e_{k}) = \begin{cases} 
                                      t_{\mathfrak{m}}( \alpha( \mathfrak{m}), e_{k} ) 
						&\textit{ if }\mathfrak{m} \textit{ is not a break;}\\
				 ( \alpha( \mathfrak{m}), e_{k+1} ) 
						&\textit{ if }\mathfrak{m} \in B_{\omega}, k \neq n, z_{k+1} =x;\\
				f( \alpha( \mathfrak{m}), e_{1} ) 
						&\textit{ if }\mathfrak{m} \in B, k = n, z_{1} =x;\\
				0 &\textit{ otherwise;} 
				\end{cases}
\]
\[
Y(\mathfrak{m},e_{k}) = \begin{cases} 
                                      ( \alpha ^{-1}( \mathfrak{m}), e_{k} ) 
						&\textit{ if }\alpha^{-1}(\mathfrak{m}) \textit{ is not a break;}\\
				 ( \alpha ^{-1}( \mathfrak{m}), e_{k-1} ) 
						&\textit{ if }\alpha^{-1}((\mathfrak{m} )\in B, k \neq 1, z_{k} =y;\\
				f( \alpha( \mathfrak{m}), e_{n} ) 
						&\textit{ if }\alpha^{-1}(\mathfrak{m}) \in B, k = 1, z_{1} =y;\\
				0 &\textit{ otherwise;} 
				\end{cases}
\]
The module $V_1(\omega, w,f)$ is irreducible if $w =x^{m}$ or $w=y^{m}$. 
The extension of this action to an action of $D$ follows the same pattern as in the previous case 
 $V_1(\omega ,j, w)$, section \ref{1Vomegajw}. That is,

\noindent
\textbf{Case 1:}  $X (v) = c \alpha (v) \neq 0$.  Then, set $Y_1(\alpha (v)) = \frac{(q\sigma -1) (v)}{c}$. 

\noindent
\textbf{Case 2:} For $\mathfrak{m} \in B_{\omega}$, 
$X(v) =0$,
$Y (\alpha (v)) = v$ and
$q\sigma (v) =v$.  In this case set $Y_1(\alpha (v)) =dv$ for any $d\in \K$.

\noindent
\textbf{Case 3:}  For $\mathfrak{m} \in B_{\omega}$, 
$X(v) =0$,
$Y (\alpha (v)) = v$ and
$q\sigma (v) \neq v $. In this case this $A_1$-module cannot be extended to a $D$-module
since we require $Y_1X (v) = (q\sigma -1)(v)$.

When one can extend to a $D$-structure, the extended $D$-module is also indecomposable. 

\begin{remark}\label{A1omegajw-indec-no-D-str}
\begin{itemize}
\item
This module is irreducible as a $D$ module
if and only if $X =0 $ at at most one basis vector and $Y, Y_1$ 
are not simultaneously zero. 
\item
Not every $A_1$-module has a $D$-module structure.
\item In Case 2 above, we could have $Y_1(\alpha (v)) =dw$ for some other $w$ with $X(w)=0, q\sigma (w)=w$
	thereby giving us yet another extension of an $A_1$-module to a $D$-module.
\end{itemize}
\end{remark}

\section{Irreducible $D$-modules}\label{sect-D-irreducible}

\begin{proposition}\label{quiver}
All nonzero weight-spaces of an irreducible weight $D$-module are $1$-dimensional.
\end{proposition}
\begin{proof}

Consider the following quiver $Q$

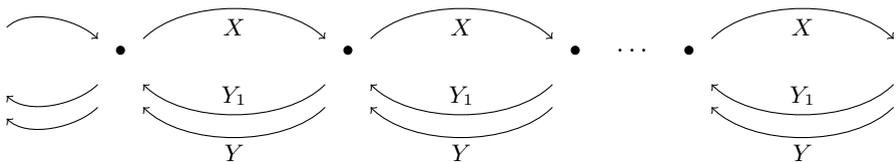
\begin{center}
\begin{tikzpicture}[scale=1.5]
\coordinate  (A) at (0,0);
\coordinate  (B) at (2,0);
\coordinate  (E) at (4,0);

\filldraw[black](E) circle (1pt);
\filldraw[black](A) circle (1pt);
\filldraw[black](B) circle (1pt);

\draw[->] (-1,0.2) to [distance=0.25cm](-0.2,0.1);
\draw[->] (-0.2,-0.3)to [distance=-0.25cm] (-1,-0.4) ;
\draw[->] (-0.2,-0.5)to [distance=-0.25cm] (-1,-0.6) ;
\draw[->] (0.2,0.1) to [distance=0.5cm](1.8,0.1);
\draw[->] (1.8,-0.3) to [distance=-0.5cm](0.2,-0.3);
\draw[->] (1.8,-0.5) to [distance=-0.5cm](0.2,-0.5);

\node (xa) at (1,0.2) {\small{$X$}};
\node (y1b) at (1,-0.4) {\small{$Y_1$}};
\node (y1b) at (1,-0.9) {\small{$Y$}};

\draw[->] (2.2,0.1) to [distance=0.5cm](3.8,0.1);
\draw[->] (3.8,-0.3) to [distance=-0.5cm](2.2,-0.3);
\draw[->] (3.8,-0.5) to [distance=-0.5cm](2.2,-0.5);
\node (xa) at (3,0.2) {\small{$X$}};
\node (y1b) at (3,-0.4) {\small{$Y_1$}};
\node (y1b) at (3,-0.9) {\small{$Y$}};

\node (dots) at (4.5, 0) {$\ldots$};

\coordinate  (C) at (5,0);

\filldraw[black](C) circle (1pt);

\draw[->] (5.2,0.1) to [distance=0.5cm](6.8,0.1);
\draw[->] (6.8,-0.3) to [distance=-0.5cm](5.2,-0.3);
\draw[->] (6.8,-0.5) to [distance=-0.5cm](5.2,-0.5);

\node (xa) at (6,0.2) {\small{$X$}};
\node (y1b) at (6,-0.4) {\small{$Y_1$}};
\node (y1b) at (6,-0.9) {\small{$Y$}};


\end{tikzpicture}

\end{center}

subject to all the relations inherited from  $D$. Let $\K Q$ be the path algebra of $Q$.  
Fix an $i$th vertex and consider the cyclic subalgebra $Q(i,i)$ of  $\K Q$ which consists of all paths that start and end in $i$. 
It is a standard fact that the restriction of any irreducible $\K Q$-module $V$ onto $Q(i,i)$ is irreducible. But $Q(i,i)=\K[\sigma, \tau]$. 
Since $\K$ is algebraically closed then any irreducible module over $Q(i,i)$ is $1$-dimensional. 
\end{proof}

\begin{theorem}\label{irred-D-to-indecom}
Let $V$ be an irreducible $D$ weight module. Then $V$ is indecomposable as an $A_q$ as well as an $A_1$ weight module.
\end{theorem}
\begin{proof}
Let $V = \oplus_{\mathfrak{m} \in \omega} V_{\mathfrak{m}}$ be an irreducible $D$ weight module
where $\omega$  is one orbit. 

\noindent
\textbf{Case:  $X$ is injective on $V$.}
If $\omega$ is cyclic, then $V$ is finite dimensional and every nonzero weight space is one dimensional. Let $v_0 \in \mathfrak{m}_0 \in \omega, v_0 \neq 0$. Then $X^i (v_0) \in \alpha^i (\mathfrak{m}_0)$ and $X^i (v_0)\neq 0$ for every $i\geq 0$. That is,  if $\omega = \{ \alpha^i (\mathfrak{m}_0 )\mid 0\leq i \leq r\}$, then $V$ has basis $\{ X^i (v_0) \mid 0\leq i \leq r \}$, and 
the only $X$-invariant subspaces are $\{ 0\}$ and $V$. That is, $V$ is indecomposable as an $A_1$ and 
as an $A_q$ module. 

If $\omega$ is linear
and suppose there exists a $v_0 \in V, v_0 \neq 0$ such that the preimage $X^{-1}(v_0) = \emptyset$. 
If  $Y(v_0) = w \neq 0$, then $XY (v_0) = (\tau -1)(v_0)$. If $(\tau -1)(v_0) = cv_0, c\neq 0$, then 
$X(\frac{w}{c}) = v_0$ contradicting the fact that 
$X^{-1}(v_0) = \emptyset$. Thus, $XY (v_0) = (\tau -1)(v_0) = 0$. If $w\neq 0$, then $X(w)=0$
contradicting the injectivity of $X$. Thus, $Y(v_0)=0$. Similar argument gives us $Y_1(v_0)=0$.  
In other words, the span of $\{ X^i (v_0) \mid i\geq 0 \}$ is a $D$-submodule of $V$, and hence 
$V = \oplus_{i=0}^{\infty}\K X^i (v_0)$ and therefore indecomposable as $A_1$ and $A_q$ module.

Suppose the preimage $X^{-1}(v) \neq \emptyset$ for any weight vector $v$. Let $v_0 \in V_{\mathfrak{m}_0}, v_0 \neq 0$ be a weight vector.
Then $V_{\alpha^i (\mathfrak{m}_0)} = \K X^i(v_0)$ for any $i\in \Z$. Note that since $X$ is injective,
$X^{i}(v_0)$ is a singleton set for $i<0$. Thus, $V$ is indecomposable as $A_1$ and $A_q$ module.

\noindent
\textbf{Case:  $X(v_0)=0$ for a nonzero vector $v_0 \in V_{\mathfrak{m}_0}$ for some $\mathfrak{m}_0\in \omega$.}
Then $V = D\cdot v_0 = S \cdot v_0$ where $S$ is the subalgebra of $D$ generated by $Y, Y_1$ over $R$.

Suppose $| \omega| = \infty$. 
That is, $V = \oplus_{i\leq 0} V_{\alpha^i (\mathfrak{m}_0)} \neq 0$. 
Suppose $X: V_{\alpha^i (\mathfrak{m}_0)} \to V_{\alpha^{i+1} (\mathfrak{m}_0)} $ is zero on 
these one-dimensional spaces for some $i<0$. Then,  $W = \oplus _{j \leq i} V_{\alpha^j (\mathfrak{m}_0)}$
is a $D$-submodule of $V$.  Therefore, $X: V_{\alpha^i (\mathfrak{m}_0)} \to V_{\alpha^{i+1} (\mathfrak{m}_0)} $ is an isomorphism for every $i<0$. For any $A_1$ or $A_q$
submodule $W$ of $V$, if
$v_i \in W$, then $X^j(v_i) \in W, \quad \forall j\geq 0$. In particular $v_0 \in W$. 
Thus $V$ is $A_1$ and $A_q$ indecomposable.

Suppose $| \omega| =r+1< \infty$. Let $V = V_{\mathfrak{m}_r} \oplus V_{\mathfrak{m}_{r-1}}\oplus
	\ldots \oplus V_{\mathfrak{m}_0}$. 
Suppose $V_{\mathfrak{m}_i} =0$ for some $i, 0 < i \leq r$,
then $W = V_{\mathfrak{m}_{i-1} }\oplus V_{\mathfrak{m}_{i-2}}\oplus \ldots \oplus V_{\mathfrak{m}_0}$
is a $D$-submodule of $V$. Therefore, we may assume that 
every $V_{\mathfrak{m}_i} $ is one dimensional for $0\leq i \leq r$ and let $u_i$ be a basis vector for
$V_{\mathfrak{m}_i}$ for each $0\leq i \leq r$.  Note that $Y, Y_1: V_{\mathfrak{m}_r} \to 
V_{\mathfrak{m}_0}$ since $\alpha^{-1}(\mathfrak{m}_r) = \mathfrak{m}_0$. 
Since $XY(u_r) = XY_1(u_r)= 0$ and $XY = \tau -1$ and $XY_1 = \sigma -1$,
we have $\mathfrak{m}_r = (\tau -1, \sigma -1)$.  If 
$X:V_{\mathfrak{m}_i} \to V_{\mathfrak{m}_{i-1}} $ is the zero map for any $i, 0<i\leq r$, then again
$\mathfrak{m}_{i-1} = (\tau -1, \sigma -1)$ contradicting the fact that $\omega$ is a single orbit.

Therefore,  $X:V_{\mathfrak{m}_i} \to V_{\mathfrak{m}_{i-1}}$
is an isomorphism for every $0<i\leq r$. 
Then $V = \oplus_{i=0}^r \K X^i (u_r)$; 
in other words, every $X$-invariant subspace of $V$ contains $v_0$, thereby making $V$ an indecomposable $A_1$ and an $A_q$ module.
\end{proof}

\subsection{Characteristic of $\K$ is $0$}\label{subsect-D-irred-char0-alg-clo}
Here $\mathfrak{m} = (\tau -a, \sigma -b)$ for some $a,b \in \K$, with $b\neq 0$.
Note that $(q\sigma -1) \in \mathfrak{m}$ if and only if $b=\dfrac{1}{q}$
and $\tau \in \mathfrak{m}$ if and only if $a=0$. In other words, the set of 
breaks $B = \{ (\tau -a, \sigma - \frac{1}{q} ) \mid a\in \K \}$ when studying $A_q$ weight modules,
and $B = \{ (\tau , \sigma -b) \mid b\in \K \}$ when studying $A_1$ weight modules.

Fix $\mathfrak{m} = (\tau -a, \sigma -b)$.  
Then, $\alpha^k (\mathfrak{m}) = (\tau -a -k , \frac{\sigma}{q^k}-b)$. 
That is, $\omega = \{ (\tau -a -k , \frac{\sigma}{q^k}-b) \mid k \in \Z \}$.
If $\K$ is of characteristic $0$ or if $q$ is not
a root of unity, then $|\omega|=\infty$.  
Note, $(q\sigma -1) \in \alpha^k (\mathfrak{m})$  if and only if $b = q^{-k-1}$ (respectively,
$\tau \in \alpha^k (\mathfrak{m})$  if and only if $a=-k$)
 for $k\in \Z$.
That is, $|B_{\omega}| \leq 1$ since $B_{\omega} = \emptyset $ if $b \notin \{ q^i \}_{i\in \Z}$ 
or $B_{\omega} =  (\tau -a -k , q\sigma-1)$ if $b = q^{-k-1}$, while studying $A_q$-modules; or 
$B_{\omega} = \emptyset $ if $a \notin  \Z$ 
or $B_{\omega} =  (\tau , \sigma - q^kb)$ if $a=k$ while studying $A_1$-modules.

We now consider the following families:

\noindent
\textbf{Family I:} Let $V$ be an irreducible $A_q$ weight module. 

\noindent
\textbf{Case: $q$ is a root of $1$.} Here,  by Theorem (5.8) of \cite{DGO}
$V$ is
isomorphic to $V_q(\omega ,b)$ for $b \in \K^* \setminus \{ q^i \}_{i\in \Z}$.   
For any $a\in \K$, the module $V_q (\omega ,b)$ can be given an irreducible $D$-weight module structure, and we denote this $D$ module by $V_q(\omega, b,a)$. The details are given in section 
\ref{DGO-indecomposables-of-Aq-char0}, case \ref{qVomega}. 

\noindent
\textbf{Case: $q$ is not a root of $1$.} 
In this case we have more families of irreducible modules. 
Here, as an irreducibel $A_q$-module, $V$ is one of the three:
$V_q(\omega, b)$ for $b\in \K^* \setminus \{ q^i \}_{i \in \Z}$ (described in 
section \ref{DGO-indecomposables-of-Aq-char0}, case \ref{qVomega}),  
or $V_q(\omega, J,J')$ (described in section \ref{DGO-indecomposables-of-Aq-char0}
case \ref{qVJJ'}) subcases $ (3)$ and $(4)$. For any $a\in \K$, these $A_q$-modules are extended to irreducible $D$ weight modules; these extensions, denoted $V_q(\omega, b,a), V_q(\omega, J,J',a)$, are also described in the mentioned sections.

\noindent
\textbf{Family II:} Let $V$ be an irreducible $A_1$ weight module. Here, $V$ is one of the three:
$V_1(\omega,b )$  (described in section \ref{Vomega}) and 
$V_1(\omega, J,J', b)$ cases $ (3)$ and $(4)$ (described in section  \ref{VJJ'} cases $(3)$ and $(4)$). 
Since these modules are $A_1$ irreducible, they are $D$-irreducible whether or not $q$ is a root of unity.

\noindent
\textbf{Family III: Irreducible $D$ weight modules which do not arise
from the families I and II above}
A $D$ weight module is also an $A_q$ and an $A_1$ weight module. Note that an irreducible
$D$-weight module need not result in an irreducible $A_q$ or an irreducible $A_1$ weight module.
For instance,  the module $V_q( \omega,J,J',a)$ described in  \ref{qVJJ'} case (1)
is an irreducible $D$ weight module which does not arise from the two families I and II. 
But by Theorem \ref{irred-D-to-indecom}, every irreducible $D$-module results in an indecomposable $A_1$
and an indecomposable $A_q$-module.
So, we start with those modules $V$ which are indecomposable (but not irreducible) as $A_1$-modules, and which can be extended to irreducible $D$-modules while ensuring that they stay indecomposable (but not irreducible) as $A_q$-modules.

Thus, $|\omega | = \infty$, and $B_{\omega} \neq \emptyset$, where $\omega = \{ \alpha^i (\tau -a, \sigma -b) \}$.  Here, refer to section \ref{VJJ'}.  We see $B_{\omega} \neq \emptyset$ implies that $a\in \Z$. 
We obtain irreducible $D$-modules in two ways. One is described in case (1), for $b \in \K \setminus \{ q^{a-1} \}$. The other, and its generalization, is described in case (2) for $b = q^{a-1}$. 
Note that in these cases, $q$ may or may not be a root of unit (refer to \ref{qpVJJ'}).

\begin{theorem}\label{alg.cl.char0.Virred}
When characteristic of $\K$ is $0$ the irreducible $D$-modules are as described in
families I, II, and III.
\end{theorem}

\begin{remark}
Not every indecomposable $A_1$-module can be given a $D$-module structure, (see Remark 
\ref{A1-indec-no-D-str}, section \ref{VJJ'}).
Likewise, not every indecomposable $A_q$-module can be given a $D$-module structure, (see Remark 
\ref{Aq-indec-no-D-str}, section \ref{qVJJ'}, and Remark \ref{Aqp-indec-no-D-str} , section \ref{qpVJJ'}).
\end{remark}

\begin{remark}
The key point in Theorem \ref{irred-D-to-indecom} is Proposition \ref{quiver}; that is, 
every weight-space of an irreducible $D$ weight module is one dimensional. Note that an indecomposable $D$ weight module need not satisfy this property. For example,  let $V (c,d) = \left( \oplus_{i \leq -1} \K u_i \right) \oplus \left( \oplus_{i \leq -1} \K w_i \right)  \oplus
	\left( \oplus_{i \geq 0} \K v_i \right) $ be given a $D$-module structure as follows:
\[
\sigma (x) = \begin{cases}
		\frac{x}{q^i} &\textit{ if } x=u_{-i},\\
		\frac{x}{q^i} &\textit{ if } x=w_{-i},\\
		q^i x &\textit{ if } x=v_{i};
	\end{cases}
\quad
X (x) = \begin{cases}
		u_{i+1}  &\textit{ if } x=u_{i}, i\neq -1,\\
		w_{i+1} &\textit{ if } x=w_{i}, i \neq -1,\\
		0 &\textit{ if } x=w_{-1} \textit{ or } x= u_{-1},\\
		v_{i+1} &\textit{ if } x=v_{i};
	\end{cases}
\]
\[
\tau (x) = \begin{cases}
		-ix  &\textit{ if } x=u_{-i},\\
		-ix &\textit{ if } x=w_{-i},\\
		ix &\textit{ if } x=v_{i};
	\end{cases}
	\quad
Y (x) = \begin{cases}
		(i-1) u_{i-1}  &\textit{ if } x=u_{i},\\
		(i-1) w_{i-1} &\textit{ if } x=w_{i}, \\
		(i-1) v_{i-1} &\textit{ if } x=v_{i}, i\neq 0,\\
		cu_{-1} &\textit{ if } x=v_0;		
	\end{cases}
\]
\[
Y_1 (x) = \begin{cases}
		(q^{i-1} -1)u_{i-1}  &\textit{ if } x=u_{i},\\
		(q^{i-1} -1) w_{i-1} &\textit{ if } x=w_{i}, \\
		(q^{i-1} -1) v_{i-1} &\textit{ if } x=v_{i}, i\neq 0,\\
		dw_{-1} &\textit{ if } x=v_0;		
	\end{cases}
\]

One can see that $V(c,d), c,d \neq 0$ is an indecomposable $D$-weight module, with weight-spaces 
$V_{\mathfrak{m}_i} = \K u_i \oplus \K w_i$ for $i\leq -1$ where
$\mathfrak{m}_i $ is the maximal ideal $ (\tau - i, \sigma - \frac{1}{q^i})$, and
$V_{\mathfrak{m}_i} = \K v_i$ for $i\geq 0$. The picture for this is as follows:


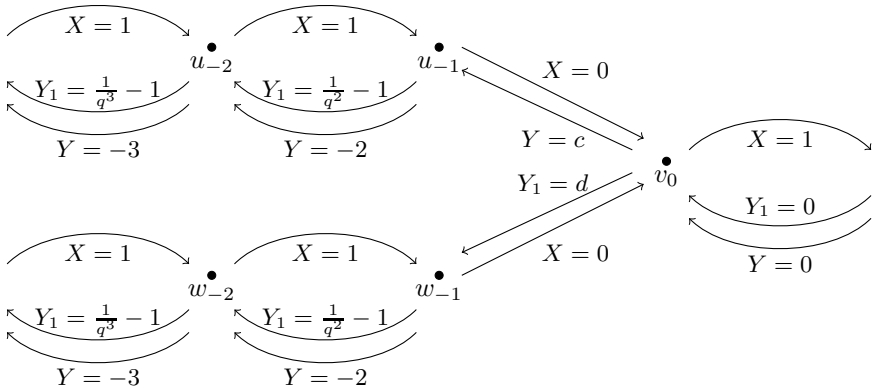
\begin{center}
\begin{tikzpicture}[scale=1.5]
\coordinate [label=below:$w_{-2}$] (B) at (2,0);
\coordinate [label=below:$w_{-1}$] (C) at (4,0);

\filldraw[black](B) circle (1pt);
\filldraw[black](C) circle (1pt);

\draw[->] (0.2,0.1) to [distance=0.5cm](1.8,0.1);
\draw[->] (1.8,-0.3) to [distance=-0.5cm](0.2,-0.3);
\draw[->] (1.8,-0.5) to [distance=-0.5cm](0.2,-0.5);
\node (xa) at (1,0.2) {\small{$X=1$}};
\node (y1b) at (1,-0.4) {\small{$Y_1=\frac{1}{q^3}-1$}};
\node (y1b) at (1,-0.9) {\small{$Y=-3$}};

\draw[->] (2.2,0.1) to [distance=0.5cm](3.8,0.1);
\node (xb) at (3,0.2) {\small{$X=1$}};
\draw[->] (3.8,-0.3) to [distance=-0.5cm](2.2,-0.3);
\draw[->] (3.8,-0.5) to [distance=-0.5cm](2.2,-0.5);
\node (y1b) at (3,-0.4) {\small{$Y_1=\frac{1}{q^2}-1$}};
\node (y1b) at (3,-0.9) {\small{$Y=-2$}};


\coordinate [label=below:$u_{-2}$] (E) at (2,2);
\coordinate [label=below:$u_{-1}$] (F) at (4,2);

\filldraw[black](E) circle (1pt);
\filldraw[black](F) circle (1pt);

\draw[->] (0.2,2.1) to [distance=0.5cm](1.8,2.1);
\draw[->] (1.8,1.7) to [distance=-0.5cm](0.2,1.7);
\draw[->] (1.8,1.5) to [distance=-0.5cm](0.2,1.5);
\node (xa) at (1,2.2) {\small{$X=1$}};
\node (y1b) at (1,1.6) {\small{$Y_1=\frac{1}{q^3}-1$}};
\node (y1b) at (1,1.1) {\small{$Y=-3$}};

\draw[->] (2.2,2.1) to [distance=0.5cm](3.8,2.1);
\node (xb) at (3,2.2) {\small{$X=1$}};
\draw[->] (3.8,1.7) to [distance=-0.5cm](2.2,1.7);
\draw[->] (3.8,1.5) to [distance=-0.5cm](2.2,1.5);
\node (y1b) at (3,1.6) {\small{$Y_1=\frac{1}{q^2}-1$}};
\node (y1b) at (3,1.1) {\small{$Y=-2$}};

\coordinate [label=below:$v_{0}$] (G) at (6,1);

\filldraw[black](G) circle (1pt);

\draw[->] (6.2,1.1) to [distance=0.5cm](7.8,1.1);
\draw[->] (7.8,0.7) to [distance=-0.5cm](6.2,0.7);
\draw[->] (7.8,0.5) to [distance=-0.5cm](6.2,0.5);
\node (xa) at (7,1.2) {\small{$X=1$}};
\node (y1b) at (7,0.6) {\small{$Y_1=0$}};
\node (y1b) at (7,0.1) {\small{$Y=0$}};

\draw[->] (4.2,2) -- (5.8,1.2);
\node (xb) at (5.2,1.8) {\small{$X=0$}};

\draw[->] (4.2,0) -- (5.8,0.8);
\node (xb) at (5.2,0.2) {\small{$X=0$}};

\draw[->] (5.7,1.1) -- (4.2,1.8) ;
\node (xb) at (5,1.2) {\small{$Y=c$}};

\draw[->] (5.7,0.9) -- (4.2,0.2) ;
\node (xb) at (5,0.8) {\small{$Y_1=d$}};

\end{tikzpicture}

\end{center}

\end{remark}

\subsection{Characteristic of $\K$ is nonzero
and $q$ is not a root of $1$}\label{subsect-D-irred-charp-alg-clo-q-not-root}
Here again we go through three families: Family I of irreducible $A_q$ modules, Family II of irreducible $A_1$ modules, and
Family III of indecomposable $A_q$ and $A_1$ modules which are not irreducible in either case. 

\noindent
\textbf{Family I:}
Here, as an irreducible $A_q$-module, $V$ is one of the three:
$V_q(\omega, b)$ for $b\in \K^* \setminus \{ q^i \}_{i \in \Z}$ (described in 
section \ref{DGO-indecomposables-of-Aq-char0}, case \ref{qVomega}),  
or $V_q(\omega, J,J')$ (described in section \ref{DGO-indecomposables-of-Aq-char0}
case \ref{qVJJ'}) subcases $ (3)$ and $(4)$. For any $a\in \K$, these $A_q$-modules are extended to irreducible $D$ weight modules; these extensions, denoted $V_q(\omega, b,a), V_q(\omega, J,J',a)$, are also described in the mentioned sections.
When $a\in \K \setminus \Z_p$, then the resulting $D$-modules are also $A_1$-modules. If $a\in \Z_p$
then the resulting modules are indecomposable $A_1$-modules. 

\noindent
\textbf{Family II:} Let $V$ be an irreducible $A_1$ weight module. Here, $V$ is one of the three:
$V_1(\omega,b, a)$  (described in section \ref{Vomega}) and 
$V_1(\omega, J, J')$ (described in section \ref{VJJ'p}).
These are $D$-irreducible weight modules.  

\noindent
\textbf{Family III: Irreducible $D$ weight modules which do not arise
from the families I and II above}
Here we start with indecomposable $A_q$-modules and extend them to irreducible $D$-modules while
keeping $A_1$-structure indecomposable, although not irreducible. 
These appear in two situations of section \ref{qVJJ'}: One is described in case (1) with $a\neq 0$; the other
is descibed in case (2) and its generalization with $c=0, d\neq 0$. 

\begin{theorem}\label{alg.cl.q-not-root.Virred}
When characteristic of $\K$ is not zero, and $q$ is not a root of $1$ the irreducible $D$-modules are as described in families I, II, and III.
\end{theorem}

\subsection{Characteristic of $\K$ is nonzero
and $q$ is a root of $1$}\label{subsect-D-irred-charp-alg-clo-q-root}
In this case, we see that $| \omega| <\infty$.  Thus, we may describe families
I and II as before. 

\noindent
\textbf{Family I:} 
When $B_{\omega}=\emptyset$, we get 
the module $V_q(\omega ,f,b)$ for $b\notin \{ q^i \}_{i\in \Z}$ and $f\in \K^*$ described in section
\ref{qVomegaf}.  This is an irreducible $A_q$-module and extends to an irreducible $D$-module.
When $B_{\omega} \neq \emptyset$, we have two families: The first family is 
the module $V_q(\omega ,j, w)$ which is irreducible for any $j\in \Z_m$ and $w$ the empty word;
this module is described described in section \ref{qVomegajw}. While not every $V_q(\omega ,j, w)$ can be
extended to a $D$-module, when $w$ is the empty word, this module can be extended to a $D$-module which will be irreducible. The second family is 
the module $V_q(\omega ,w,f)$ which is irreducible for the word $w=x^m$ or $w=y^m$; this module is described in section \ref{qVomegawf}. When $w=x^m$, this module can be extended to 
a $D$-module. But if $w=y^m$, then this module may or may not be extended to a $D$-module. 

\noindent
\textbf{Family II:}
As in the preceding paragraph, we have three families of irreducible $A_1$ modules which may be extended
to an irreducible $D$-module. 
The module $V_1(\omega ,f,a)$ for $a\notin \Z_p$ is described in section \ref{1Vomegaf}
and arises when $B_{\omega} = \emptyset$.
This is an irreducible $A_1$ module, and can be extended to a $D$-module. 

The module $V_1(\omega ,j, w)$ is described in section \ref{1Vomegajw}
and arises when $B_{\omega} \neq \emptyset$. This is an irreducible $A_1$-module
if and only if $w$ is the empty word, and it can then be extended to a $D$-module. 

The module $V_1(\omega ,w,f)$ is described in section \ref{1Vomegawf}  and arises when
$B_{\omega} \neq \emptyset$. This is an irreducible $A_1$ module if and only if $w=x^m$ or $x=y^m$.
When $w=x^m$, this module can be extended to an irreducible $D$-module. When $x=y^m$, this module
may or may not be extended to a $D$-module. 

\noindent
\textbf{Family III:}
We start with indecomposable $A_1$-modules (but not irreducible) and extend in such a fashion
that the resultant modules are indecomposable $A_q$-modules (but not irreducible)
and are irreducible as $D$ modules.  Thus, we look at the case where $| \omega|<\infty$ 
and $B_{\omega} \neq \emptyset$. As has been explained in the proof of Theorem \ref{irred-D-to-indecom}
 the kernel of $X$ is one-dimensional. By Proposition \ref{quiver} every weight space is also one-dimensional.
In other words, when we consider module $V_q(\omega, j, w)$ ( section \ref{qVomegajw})
or $V_1(\omega, j, w)$ (section \ref{1Vomegajw}) we need to consider $w =x^m$ where 
$m=|B_{\omega}|$. Both these families may or may not be extended to families of $D$-modules. If 
extended, the extended families are irreducible if and  only if $Y$ and $Y_1$ are not simultaneously zero except at $(\mathfrak{m}_0, e_0)$. 

Similarly, we consider those modules $V_q(\omega ,w,f)$ (section \ref{1Vomegawf})
and $V_1(\omega ,w,f)$ (section \ref{1Vomegawf}). These modules are irreducible as $A_q$, and respectively $A_1$ modules if and only if $w=x^m$ or $w=y^m$. Since we require the dimension of
the kernel of $X$ to be at most $1$, and the dimension of every weight space be equal to $1$, 
we are forced to require that $w$ be a word of degree $m$ in variables
$x,y$ with degree in $y$ being exactly $1$. Both these families may or may not be extended to families of $D$-modules. If 
extended, the extended families are irreducible if and  only if $Y$ and $Y_1$ are not simultaneously zero
except at the same edge as the edge $X=0$. 

\begin{theorem}\label{alg.cl.finite.orbit}
When characteristic of $\K$ is nonzero and $q$ is a root of $1$, the irreducible $D$ weight modules
are those described in
families I ,II, and III.
\end{theorem}

\subsection{A family of indecomposable $D$-modules which are decomposable as $A_q$ and $A_1$ modules.}\label{indecomposablefamily}

Throughout this section we assume that the characteristic of $\K$ is $p >0$  and $q$ is a root of $1$
of order $p$.

In this section we present a family of finite dimensional 
indecomposable $D$-modules which are decomposable as $A_q$ and 
as $A_1$ modules. First, we present an example.

\subsubsection{Characteristic of $\K$ is $3$, and $q^3=1$.} \label{counter-example-to-thm1}

Let $\mathfrak{m}_0=(\tau -1, \sigma -1)$. Then
$\alpha^i (\mathfrak{m}_0) = \{ \mathfrak{m}_0, \mathfrak{m}_1, \mathfrak{m}_2 \}$ with
$\mathfrak{m}_i =(\tau -i-1, \sigma -q^i)$.
Fix $a,b,c,d \in \K^*$.
Let $V = \oplus_{i=0}^2 V_i$ where each 
$V_i $ is a four dimensional $\K_{m_i}$ module with basis
$\{ (i, e_1), (i, e_2), (i, e_3), (i, e_4)  \}$. 
This $D$ module can be visualized as follows:
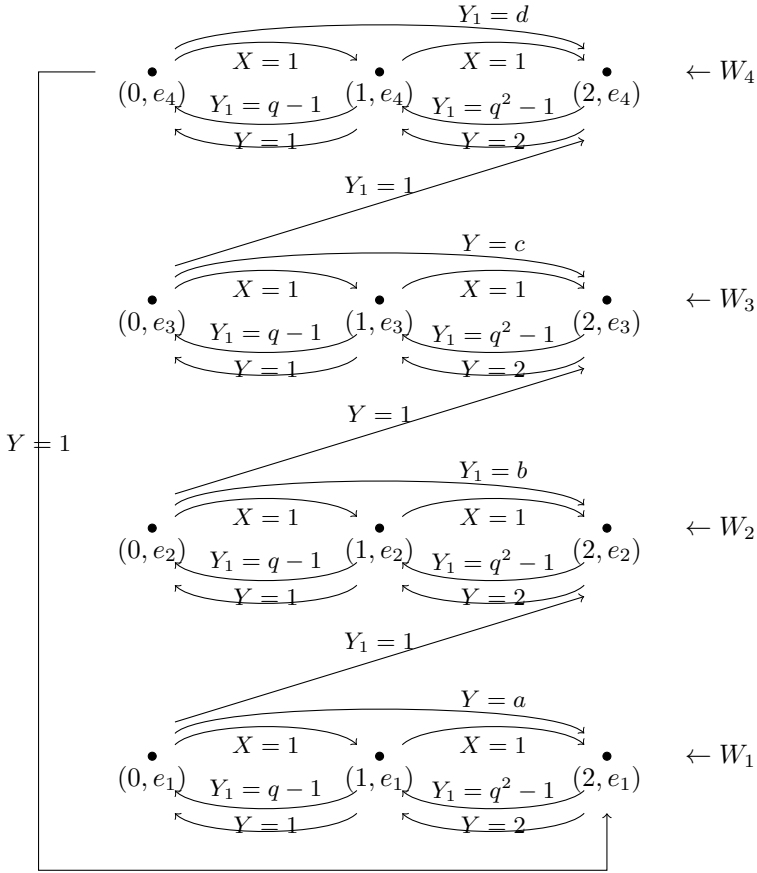
\begin{center}
\begin{tikzpicture}[scale=1.5]

\node (J) at (0,5.8) {$(0,e_4)$};
\node (K) at (2,5.8) {$(1,e_4)$};
\node (L) at (4,5.8) {$(2,e_4)$};

\filldraw[black](0,6) circle (1pt);
\filldraw[black](2,6) circle (1pt);
\filldraw[black](4,6) circle (1pt);

\draw[->] (0.2,6.1) to [distance=0.3cm](1.8,6.1);
\draw[->] (1.8,5.7) to [distance=-0.3cm](0.2,5.7);
\draw[->] (1.8,5.5) to [distance=-0.3cm](0.2,5.5);
\node (xa) at (1,6.1) {\small{$X=1$}};
\node (y1b) at (1,5.7) {\small{$Y_1=q-1$}};
\node (y1b) at (1,5.4) {\small{$Y=1$}};

\draw[->] (2.2,6.1) to [distance=0.3cm](3.8,6.1);
\node (xb) at (3,6.1) {\small{$X=1$}};
\draw[->] (3.8,5.7) to [distance=-0.3cm](2.2,5.7);
\draw[->] (3.8,5.5) to [distance=-0.3cm](2.2,5.5);
\node (y1b) at (3,5.7) {\small{$Y_1=q^2-1$}};
\node (y1b) at (3,5.4) {\small{$Y=2$}};


\node (G) at (0,3.8) {$(0,e_3)$};
\node (H) at (2,3.8) {$(1,e_3)$};
\node (I) at (4,3.8) {$(2,e_3)$};

\filldraw[black](0,4) circle (1pt);
\filldraw[black](2,4) circle (1pt);
\filldraw[black](4,4) circle (1pt);

\draw[->] (0.2,4.1) to [distance=0.3cm](1.8,4.1);
\draw[->] (1.8,3.7) to [distance=-0.3cm](0.2,3.7);
\draw[->] (1.8,3.5) to [distance=-0.3cm](0.2,3.5);
\node (xa) at (1,4.1) {\small{$X=1$}};
\node (y1b) at (1,3.7) {\small{$Y_1=q-1$}};
\node (y1b) at (1,3.4) {\small{$Y=1$}};

\draw[->] (2.2,4.1) to [distance=0.3cm](3.8,4.1);
\node (xb) at (3,4.1) {\small{$X=1$}};
\draw[->] (3.8,3.7) to [distance=-0.3cm](2.2,3.7);
\draw[->] (3.8,3.5) to [distance=-0.3cm](2.2,3.5);
\node (y1b) at (3,3.7) {\small{$Y_1=q^2-1$}};
\node (y1b) at (3,3.4) {\small{$Y=2$}};


\node (D) at (0,1.8) {$(0,e_2)$};
\node (E) at (2,1.8) {$(1,e_2)$};
\node (F) at (4,1.8) {$(2,e_2)$};

\filldraw[black](0,2) circle (1pt);
\filldraw[black](2,2) circle (1pt);
\filldraw[black](4,2) circle (1pt);

\draw[->] (0.2,2.1) to [distance=0.3cm](1.8,2.1);
\draw[->] (1.8,1.7) to [distance=-0.3cm](0.2,1.7);
\draw[->] (1.8,1.5) to [distance=-0.3cm](0.2,1.5);
\node (xa) at (1,2.1) {\small{$X=1$}};
\node (y1b) at (1,1.7) {\small{$Y_1=q-1$}};
\node (y1b) at (1,1.4) {\small{$Y=1$}};

\draw[->] (2.2,2.1) to [distance=0.3cm](3.8,2.1);
\node (xb) at (3,2.1) {\small{$X=1$}};
\draw[->] (3.8,1.7) to [distance=-0.3cm](2.2,1.7);
\draw[->] (3.8,1.5) to [distance=-0.3cm](2.2,1.5);
\node (y1b) at (3,1.7) {\small{$Y_1=q^2-1$}};
\node (y1b) at (3,1.4) {\small{$Y=2$}};

\node (A) at (0,-0.2) {$(0,e_1)$};
\node (B) at (2,-0.2) {$(1,e_1)$};
\node (C) at (4,-0.2) {$(2,e_1)$};

\filldraw[black](0,0) circle (1pt);
\filldraw[black](2,0) circle (1pt);
\filldraw[black](4,0) circle (1pt);

\draw[->] (0.2,0.1) to [distance=0.3cm](1.8,0.1);
\draw[->] (1.8,-0.3) to [distance=-0.3cm](0.2,-0.3);
\draw[->] (1.8,-0.5) to [distance=-0.3cm](0.2,-0.5);
\node (xa) at (1,0.1) {\small{$X=1$}};
\node (y1b) at (1,-0.3) {\small{$Y_1=q-1$}};
\node (y1b) at (1,-0.6) {\small{$Y=1$}};

\draw[->] (2.2,0.1) to [distance=0.3cm](3.8,0.1);
\node (xb) at (3,0.1) {\small{$X=1$}};
\draw[->] (3.8,-0.3) to [distance=-0.3cm](2.2,-0.3);
\draw[->] (3.8,-0.5) to [distance=-0.3cm](2.2,-0.5);
\node (y1b) at (3,-0.3) {\small{$Y_1=q^2-1$}};
\node (y1b) at (3,-0.6) {\small{$Y=2$}};

\draw[->] (0.2,0.2) to [distance=0.4cm] (3.8,0.2);
\draw[->] (0.2, 0.3) -- (3.8,1.4);
\node (y1) at (3,0.5) {\small{$Y=a$}};
\node (y2) at (2,1) {\small{$Y_1=1$}};
\draw[->] (0.2,2.2) to [distance=0.4cm] (3.8,2.2);
\draw[->] (0.2, 2.3) -- (3.8,3.4);
\node (y3) at (3,2.5) {\small{$Y_1=b$}};
\node (y4) at (2,3) {\small{$Y=1$}};
\draw[->] (0.2,4.2) to [distance=0.4cm] (3.8,4.2);
\draw[->] (0.2, 4.3) -- (3.8,5.4);
\node (y5) at (3,4.5) {\small{$Y=c$}};
\node (y6) at (2,5) {\small{$Y_1=1$}};
\draw[->] (0.2,6.2) to [distance=0.4cm] (3.8,6.2);
\node (y5) at (3,6.5) {\small{$Y_1=d$}};
\draw[->] (-0.5,6) -- (-1,6) -- (-1,-1) -- (4,-1)--(4,-0.5);
\node (y6) at (-1,2.75) {\small{$Y=1$}};

\node (w1) at (5,0) {$\leftarrow W_1$};
\node (w2) at (5,2) {$\leftarrow W_2$};
\node (w3) at (5,4) {$\leftarrow W_3$};
\node (w4) at (5,6) {$\leftarrow W_4$};

\end{tikzpicture}

\end{center}

That is, the action of $D$ on $V$ is defined as follows:
\[
X (i, e_j) = \begin{cases}
		(i+1, e_{j}) &\textit{ if } i\neq 2,\\
		0 &\textit{ if } i= 2 ;
		\end{cases}
\]
\[
Y (i, e_j) = \begin{cases}
		i(i-1, e_{j}) &\textit{ if } i\neq 0,\\
		a(2, e_1) &\textit{ if } i=0, j= 1, \\
		(2, e_3) &\textit{ if } i=0, j= 2, \\
		c(2, e_3) &\textit{ if } i=0, j= 3, \\
		(2, e_1)&\textit{ if } i=0, j= 4;
		\end{cases}
\]
\[
Y_1 (i, e_j) = \begin{cases}
		(q^i-1)(i-1, e_{j}) &\textit{ if } i\neq 0,\\
		(2, e_2) &\textit{ if } i=0, j= 1, \\
		b(2, e_2) &\textit{ if } i=0, j= 2, \\
		(2, e_4) &\textit{ if } i=0, j= 3, \\
		d(2, e_4)&\textit{ if } i=0, j= 4.
		\end{cases}
\]
While $V$ is an indecomposable $D$-module, as an $A_1$ module $V$ decomposes as $V =( W_1\oplus W_4) \oplus (W_2\oplus W_3)$
and as an $A_q$-module $V$ decomposes as $V = (W_1\oplus W_2) \oplus (W_3\oplus W_4)$
where each $W_i$ is the three dimensional vector space spanned by $\{ (0,e_i), (1,e_i), (2,e_i) \}$
for $i=0,\ldots ,4$. 

\subsubsection{Generalization.} \label{counter-example-family}
The above example \ref{counter-example-to-thm1}
can be generalized. Let $m\geq 4$ be an even natural number
and $a_1,a_2,\ldots ,a_m \in \K^*$. 
Let $\mathfrak{m}_0 = (\tau-1, \sigma -1)$, and
$\mathfrak{m}_i = \alpha^i (\mathfrak{m}_0)$ for $0\leq i\leq p-1$. 
For each  $0\leq i\leq p-1$
let $V_i$ be the $\K_{\mathfrak{m}_i}$ with basis
$(i, e_j)$ for $1\leq j \leq m$. Let $V(m;a_1,a_2,\ldots ,a_m) = \oplus_{i=0}^{p-1}V_i$
and give it $D$-action by 
$
X (i, e_j) = \begin{cases}
		(i+1, e_{j}) &\textit{ if } i\neq p-1,\\
		0 &\textit{ if } i= p-1 ;
		\end{cases}
$
\[
Y (i, e_j) = \begin{cases}
		i(i-1, e_{j}) &\textit{ if } i\neq 0,\\
		a_{2k-1}(p-1, e_j) &\textit{ if } i=0, j= 2k-1, \\
		(p-1, e_{j+1}) &\textit{ if } i=0, j= 2k<m, \\
		(p-1, e_1)&\textit{ if } i=0, j= m;
		\end{cases}
		\]
		\[
Y_1 (i, e_j) = \begin{cases}
		(q^i-1)(i-1, e_{j}) &\textit{ if } i\neq 0,\\
		a_{2k}(p-1, e_j) &\textit{ if } i=0, j= 2k, \\
		(p-1, e_{j+1}) &\textit{ if } i=0, j= 2k-1,
		\end{cases}
\]
where $k$ is a natural number, $1\leq k \leq \frac{m}{2}$. 
This action results in an indecomposable $D$-module structure on 
$V(m;a_1,a_2,\ldots ,a_m)$. 
For each $1\leq j \leq m$, let $W_j$ be the vector space with basis
$\{ (i, e_j)\mid 0\leq i \leq p-1 \}$. 

Note, for each $k, 1\leq k \leq \frac{m}{2}$, the
vector space $W_{2k-1}\oplus W_{2k}$ is an $A_q$ module, and
viewed as an $A_q$-module, $V(m;a_1,a_2,\ldots ,a_m)$ is decomposed
into a direct sum of $A_q$-submodules: 
\[
V(m;a_1,a_2,\ldots ,a_m) = \oplus_{k=1}^{\frac{m}{2}} (W_{2k-1} \oplus W_{2k}).
\]
Similarly, for each $k, 1\leq k < \frac{m}{2}$, the vector space $W_{2k}\oplus W_{2k+1}$ 
is an $A_1$ module; in addition, the vector space $W_{m}\oplus W_1$ is an
$A_1$-module. Further, $V(m;a_1,a_2,\ldots ,a_m)$ is decomposed
into a direct sum of $A_1$-submodules: 
\[
V(m;a_1,a_2,\ldots ,a_m) = \oplus_{k=1}^{\frac{m}{2}-1} (W_{2k} \oplus W_{2k+1}) \oplus
	(W_m \oplus W_1).
\]
We thus have,

\noindent
\textbf{Example 1:} 
For any even natural number $m \geq 4$, and $a_1,a_2,\ldots ,a_m \in \K^*$, 
there is an indecomposable $D$-weight module $V(m;a_1,a_2,\ldots ,a_m)$ which is decomposable 
as an $A_q$ as well as an $A_1$ weight module into $\dfrac{m}{2}$ components.

This example leads us to the contruction of a family of indecomposable $D$ modules as follows:
Let $Y_1, Y$ be two noncommuting variables, $z\in \{ Y_1,Y \}$ and $a\in \K^*$.
As above, let $\mathfrak{m}_0 = (\tau-1, \sigma -1)$, and
$\mathfrak{m}_i = \alpha^i (\mathfrak{m}_0)$ for $0\leq i\leq p-1$. 
Let $W^{z, a} = \oplus_{i=0}^{p-1} \K v_i$ where $v_i$ is of weight $\mathfrak{m}_i$. Then $W^{z,a}$
has a $D$-module structure given by 
$X (v_i) = \begin{cases} v_{i+1} &\textit{ if }i<p-1,\\
				0  &\textit{ if }i=p-1; 
		\end{cases}$
\[
Y (v_i) = \begin{cases} iv_{i-1} &\textit{ if }i>0,\\
				av_{p-1}  &\textit{ if }i=0, \textit{ and }z=Y, \\
				0  &\textit{ if }i=0, \textit{ and }z= Y_1;
		\end{cases}
\]
\[
Y_1 (v_i) = \begin{cases} (q^i-1)v_{i-1} &\textit{ if }i>0,\\
				av_{p-1}  &\textit{ if }i=0, \textit{ and }z=Y_1, \\
				0  &\textit{ if }i=0, \textit{ and }z= Y.
		\end{cases}
\]

Fix $m$ a natural number, $a_1,a_2,\ldots ,a_m \in \K^*$,
and $w = z_1z_2\cdots z_m$ a word of length $m$ where
$z_i \in \{Y_1, Y \}$. Let $V(m; w, a_1,\ldots, a_m) = \oplus_{i=1}^m W^{z_i,a_i}$ as a vector
space. We now connect each component to give $V(m; w, a_1,\ldots, a_m)$ 
an indecomposable $D$ module structure. Namely, $V(m; w, a_1,\ldots, a_m) $
receives action of $X$ from its action on each
component $W^{z_i,a_i}$. The action of $Y$ on $W^{z_i,a_i}$ is also inherited if
$z_i=Y$. If $z_i =Y_1$, then for $i>0$, action of $Y$ on $v_i \in W^{z_i,a_i}$
is inherited; for
$v_0 \in W^{z_i,a_i}$ set 
$Y(v_0) = v_{p-1} \in W^{z_{i+1},a_{i+1}}$ when $1\leq i<m$, and set
$Y(v_0) = v_{p-1} \in W^{z_{1},a_{1}}$ when $i=m$. 
Similarly, the action of $Y_1$ on $W^{z_i,a_i}$ is also inherited if
$z_i=Y_1$. If $z_i =Y$, then for $i>0$, action of $Y_1$ on $v_i \in W^{z_i,a_i}$
is inherited;
for $v_0 \in W^{z_i,a_i}$ set 
$Y_1(v_0) = v_{p-1} \in W^{z_{i+1},a_{i+1}}$ when $1\leq i<m$, and set
$Y_1(v_0) = v_{p-1} \in W^{z_{1},a_{1}}$ when $i=m$. 

For $1\leq i \leq m$ if we view each $W^{z_i,a_i}$ as a vertex, then  $V(m; w, a_1,\ldots, a_m)$
is to be viewed as a directed cycle on $m$ vertices with an edge $Y$ (if $z_i =Y_1$) or 
$Y_1$ (if $z_i=Y$) joining vertex $W^{z_i,a_i}$ to $W^{z_{i+1},a_{i+1}}$ if $i<m$ and joining
$W^{z_m,a_m}$ to $W^{z_1,a_1}$ for $i=m$.  As an examle, for word $w = Y_1YY_1$, and 
$a_1,a_2,a_3\in \K^*$ the cycle would be

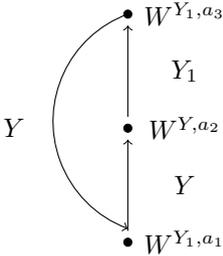
\begin{center}
\begin{tikzpicture}[scale=1.5]

\node(A) at (0.5,0){$W^{Y_1,a_1}$};
\node(B) at (0.5,1){$W^{Y,a_2}$};
\node(C) at (0.5,2){$W^{Y_1,a_3}$};
\filldraw[black](0,0) circle (1pt);
\filldraw[black](0,1) circle (1pt);
\filldraw[black](0,2) circle (1pt);

\draw[->] (0,0.1) -- (0,0.9);
\draw[->] (0,1.1) -- (0,1.9);
\draw[->] (0,2) arc (110: 250:1cm);

\node(D) at (0.5,0.5){$Y$};
\node(E) at (0.5,1.5){$Y_1$};
\node(F) at (-1,1){$Y$};

\node(G) at  (4,1){Here, every $W^{z_i,a_i}$ is a $p$-dimensional vector space.};

\end{tikzpicture}

\end{center}

\begin{theorem}\label{thm-family-pm}
Let $p$ denote the characteristic of $\K$ and assume that  $p$ is equal to the order of $q$
as a root of $1$.  For natural number $m$, $a_1,a_2,\ldots ,a_m \in \K^*$, and any
word $w= z_1z_2\cdots z_m$ with $z_i \in \{Y_1, Y \}$  
there is an irreducible $D$-weight module $V(m;w, a_1,a_2,\ldots ,a_m)$ 
of dimension $pm$. We also have:

\begin{itemize}
\item  \textbf{Case $m=1$:} Here $V(m;w, a_1,a_2,\ldots ,a_m)$ is an irreducible $A_1$ and $A_q$ module. 
\item \textbf{Case $m=1,2,3$:} 
	Here $V(m;w, a_1,a_2,\ldots ,a_m)$ is indecomposable as an $A_1$ or an $A_q$ module. 
\item \textbf{Case $m\geq 4$:} 
	If $w=Y_1^m$, then $V(m;w, a_1,a_2,\ldots ,a_m)$ is indecomposable as an $A_1$ module,
	but decomposable as an $A_q$ module. Similarly, if  $w=Y^m$, then $V(m;w, a_1,a_2,\ldots ,a_m)$ is 		         indecomposable as an $A_q$ module,
	but decomposable as an $A_1$ module. 
\item \textbf{Case $m\geq 4$:} If the degree of $w$ in $Y_1$ is $1$ then $V(m;w, a_1,a_2,\ldots ,a_m)$ is
	indecomposable as an $A_q$ module, but decomposable as an $A_1$ module.Similarly, 
	if the degree of $w$ in $Y$ is $1$ then $V(m;w, a_1,a_2,\ldots ,a_m)$ is
	indecomposable as an $A_1$ module, but decomposable as an $A_q$ module.
\item \textbf{Case $m\geq 4$:} If $w$ is of degree at least $2$ in $y$ and at least degree $2$ in $Y_1$, 
	then  $V(m;w, a_1,a_2,\ldots ,a_m)$ is decomposable as an $A_1$ and an $A_q$ module. 
\end{itemize}
\end{theorem}

While we do not have any general statements of indecomposable finite dimensional $D$-modules
we can say the following:
\begin{theorem}\label{thm-equidimension-p=|q|}
Let $V = \oplus_{\mathfrak{m} \in \omega} V_{\mathfrak{m}}$ 
be a finite dimensional $D$ weight module with support contained in orbit $\omega$. 
Then $dim (V_{\mathfrak{m}}) = dim(V_{\mathfrak{n}})$ for any $\mathfrak{m}, \mathfrak{n}
\in \omega$.
\end{theorem}
\begin{proof}
Since
the characteristic of $\K$ is $p$ and so is the order of $q$, we see that $\omega$ 
is a circular orbit of length $p$. 
Further, $\omega$ has at most one $A_q$-break and has at most one $A_1$-break. 

Suppose that $\omega$ does not have an $A_q$-break or an $A_1$-break.
Let $\mathfrak{m} \in \omega$ be in the support of $V$, and consider
$X: V_{\mathfrak{m}} \to V_{\alpha (\mathfrak{m})}$.  Suppose $v\in  V_{\mathfrak{m}}$
is in the kernel of $X$.  That is, $YX (v) = \tau (v) =0$ and $Y_1X(v) = (q\sigma -1)(v)=0$. 
That is, $\mathfrak{m} = (\tau, \sigma -\frac{1}{q})$ which implies that 
$\mathfrak{m}$ is an $A_q$ as well as an $A_1$ break.  That is, $X$ is injective for every
$\mathfrak{m}$. Since $\omega$ is a circular orbit, 
$X: V_{\mathfrak{m}} \to V_{\alpha (\mathfrak{m})}$ 
is an isomorphism of vectorspaces for every
$\mathfrak{m}$.

Now suppose that $\omega$ has an $A_q$-break or an $A_1$-break. Without loss of generality,
suppose that $\omega$ has an $A_q$-break, and since it is unique, let us name it $\mathfrak{m}_{p-1}$. 
Let $\omega = \{ \mathfrak{m}_0, \mathfrak{m}_1, \ldots, \mathfrak{m}_{p-1} \}$ 
where $\mathfrak{m}_i = \alpha^i (\mathfrak{m}_0)$. 
As $\mathfrak{m}_i$ is not a break for any $i<p-1$, then 
by the argument in the preceding paragraph, $X: V_{\mathfrak{m}_i} \to 
V_{\mathfrak{m}_{i+1}}$ is injective for every $i<p-1$. 

Suppose for some $i<p-1$, $X: V_{\mathfrak{m}_i} \to 
V_{\mathfrak{m}_{i+1}}$ is not surjective. Then there is a $v\in V_{\mathfrak{m}_{i+1}}$
not in $X(V_{\mathfrak{m}_{i}})$. Now, $Y_1(v) \in V_{\mathfrak{m}_{i}}$. Note,
$XY_1(v) = (\sigma -1)(v)$. If $(\sigma -1)(v) \neq 0$, then $v\in X(V_{\mathfrak{m}_{i}})$
contradicting the choice of $v$. Hence, $(\sigma -1)(v) =0$.  In other words, $v\in \mathfrak{m}_0$. 
In other words, $X: V_{\mathfrak{m}_i} \to 
V_{\mathfrak{m}_{i+1}}$ is surjective for every $i, 0\leq i <p-1$. 
 We have thus proved the theorem.
\end{proof}
\begin{remark}
Note that by Proposition \ref{quiver}, every weight space of an irreducible $D$-module is of dimension $1$
over an algebraically closed field, irrespective of its characteristic. 
Theorem \ref{thm-equidimension-p=|q|} seems to lead to a generalization of Proposition \ref{quiver} for $D$-modules not necessarily irreducible. Yet,
Theorem \ref{thm-equidimension-p=|q|} is not true for weight $D$-modules when the characteristic $\K$
is not equal to the order of $q$. For instance, suppose characteristic of $\K$ is $3$ and $q^2=1$.
Then, a cyclic orbit is of length $6$.  Let $\mathfrak{m}_1 = (\tau -1, \sigma -1)$
and $\omega = \{ \mathfrak{m}_i=\alpha^i (\mathfrak{m}_1) \mid 0\leq i \leq 5 \}$. 
Let $V_{\mathfrak{m}_i}$ be one-dimensional with basis $\{ v_i \}$ for $1\leq i \leq 3$ 
and zero-dimensional otherwise. Assign it a $D$-module structure by

\begin{center}
\begin{tabular}{lllll}
$\tau (v_1) = v_1$, &&$\sigma (v_1)=v_1$,&&$X(v_1) = v_2$,\\
$\tau (v_2) = 2v_2$, &&$\sigma (v_2)=qv_2$,&&$X(v_2) = v_3$, \\
$\tau (v_3) = 0$, &&$\sigma (v_3)=v_3$,&&$X(v_3) = 0$,
\end{tabular}
\end{center}

\begin{center}
\begin{tabular}{lll}
$Y(v_1) =0$, &&$Y_1(v_1)=0$,\\
$Y(v_2) =v_1$, &&$Y_1(v_2)=(q-1)v_1$,\\
$Y(v_3) =2v_2$, &&$Y_1(v_3)=0$.
\end{tabular}
\end{center}

\end{remark}

\subsection{Isomorphisms among various families of $D$-irreducible weight modules.} \label{D-irred-iso-fam}
The following isomorphisms are noted:
\begin{enumerate}
\item When $b \in \K^* \setminus \{ q^i\}_{i\in \Z}$ and $a\in \K \setminus \Z$, then
	$V_q(\omega ,b,a) \cong V_1(\omega ,a,b)$. These modules are described in section
	\ref{DGO-indecomposables-of-Aq-char0}, case \ref{qVomega}, and section
	\ref{DGO-indecomposables-of-A1-char0}, case \ref{Vomega}, respectively.
\item  When $q$ is not a root of unity, $b = q^k$ for some $k\in Z$ and $a\in \K \setminus \Z$, then
	$ V_1(\omega ,a,b) \cong V_q(\omega ,J,J'a)$ 
	 where the two  are described in section
	 \ref{DGO-indecomposables-of-A1-char0}, case \ref{Vomega}, and section
	\ref{DGO-indecomposables-of-Aq-char0}, case \ref{qVJJ'}, subcase (1) respectively.
\item   When $q$ is a root of unity, $b = q^k$ for some $k\in Z$ and $a\in \K \setminus \Z$, then
	$ V_1(\omega ,a,b) \cong V_q(\omega ,J,J)$ 
	 where the two  are described in section
	 \ref{DGO-indecomposables-of-A1-char0}, case \ref{Vomega}, and section
	\ref{DGO-indecomposables-of-Aq-char0}, case \ref{qpVJJ'}, case (1) respectively.
\item When $b \in \K^* \setminus \{ q^i\}_{i\in \Z}$ and $a\in  \Z$, then
	$V_q(\omega ,b,a) \cong V_1(\omega ,J,J', b)$. These modules are described in section
	\ref{DGO-indecomposables-of-Aq-char0}, case \ref{qVomega}, and section
	\ref{DGO-indecomposables-of-A1-char0}, case \ref{VJJ'}, subcase (1) respectively.
\item When $a\notin \Z_p$ and $q\notin \{ q^i \}_{i\in \Z}$, we have 
	$V_q(\omega ,f,b, a) \cong V_1(\omega ,f,a, b)$.
\item  When $w=\epsilon$, the emtpy word,  $V_q(\omega, j, \epsilon)$ is isomorphic to
	$V_1(\omega ,j, \epsilon)$ as $D$-modules; refer to sections \ref{qVomegajw} and \ref{1Vomegajw}
	for details.

\end{enumerate}


\section*{Acknowledgements}
V.F. is supported in part by the CNPq grant (301320/2013-6) and by the
Fapesp grant (2010/50347-9).
U.I. gratefully acknowledges hospitality from
 Insitute of Mathematics and Statistics, University of S\~ ao Paulo. 
She is supported in part by the Fapesp grant (2013/13970-8).


\bibliographystyle{amsplain}

\begin{thebibliography}{99} 
	\bibitem{BBF}
	V. Bekkert, V., Benkart, G.,  Futorny, V., 
	\emph{Weight modules for Weyl algebras},
	in ``Kac-Moody Lie algebras and related topics (Proc. Ramanujan internat.
		symp., Chennai, India, 2002)", Contemp. Math. 343, Amer. Math. Soc.,
		Providence, RI, (2004) 17-42.
	\bibitem{DGO}
	 Drozd, Yu. A., Guzner, B. L., Ovsienko, S. A., 
	\emph{Weight modules over generalized Weyl algebras,}
	J. Algebra 184 (1996) 491-504.
	\bibitem{IJ}
 	 Iyer, U. N.,  Jordan, D. A.,
	  \emph{Noetherian algebras of quantum differential operators},
  		(Work in progress)
	\bibitem{IM}
	 Iyer, U. N.,  McCune, T. C.,
 	 \emph{Quantum differential operators on $\K [x]$},
 		 Internat. J. of Math., Vol. 13, No.4 (2002) 395-413.
	\bibitem{LR}
  		Lunts, V.,  Rosenberg, A.,
  		\emph{Differential operators on noncommutative rings},
  		Selecta Math.(N.S) \textbf{3},
  		335--359 (1997).
\end{thebibliography}

\end{document}